\documentclass[final]{siamart}
\input{tocdef}
\DeclareMathAlphabet      {\mathup}{OT1}{\familydefault}{m}{n}
\usepackage{color,amsmath,amssymb,mathrsfs}
\usepackage{graphicx}
\usepackage{algorithmic,algorithm}
\usepackage{tikz,pgfplots}
\usepackage{showkeys,cite}
\usepackage{mathtools,bm}

\DeclareMathAlphabet{\mathpzc}{OT1}{pzc}{m}{it}
\DeclareSymbolFont{bbold}{U}{bbold}{m}{n}
\DeclareSymbolFontAlphabet{\mathbbold}{bbold}
\definecolor{grassgreen}{RGB}{92,135,39}
\newtheorem{remark}{Remark}[section]

\newcommand{\qedhere}{\proofbox}
\newcommand{\hilb}{\mathcal{V}}

\renewcommand{\vec}[1]{{\mathchoice
                     {\mbox{\boldmath$\displaystyle{#1}$}}
                     {\mbox{\boldmath$\textstyle{#1}$}}
                     {\mbox{\boldmath$\scriptstyle{#1}$}}
                     {\mbox{\boldmath$\scriptscriptstyle{#1}$}}}}

\newcommand{\trace}{\mathsf{tr}\,}
\newcommand{\diag}{\mathsf{diag}\,}
\newcommand{\eps}{\varepsilon}

\newcommand{\norm}[1]{\left\| {#1} \right\|}

\newcommand{\ip}[2]{{\left\langle {#1}, {#2} \right\rangle}}
\newcommand{\mip}[2]{\left\langle{#1}, {#2}\right\rangle_{\!\scriptscriptstyle{\boldsymbol{M}}}}
\newcommand{\mat}[1]{\boldsymbol{\mathrm{#1}}}
\newcommand\restr[2]{{
  \left.\kern-\nulldelimiterspace % automatically resize the bar with \right
  {#1}\vphantom{\big|} \right|_{#2}}}

\newcommand{\Rnm}{\hilb_\Nm}
\newcommand{\R}{\mathbb{R}}

\newcommand{\A}{\mathcal{A}}

\newcommand{\C}{\mathcal{C}}
\newcommand{\D}{\mathcal{D}}

\renewcommand{\P}{\mathcal{P}}

\newcommand{\ave}{\operatorname{E}}

\newcommand{\GM}[2]{\mathcal{N}\!\left( {#1}, {#2}\right)}

\newcommand*\oline[1]{%
  \vbox{%
    \hrule height 0.6pt%                  % Line above with certain width
    \kern0.25ex%                          % Distance between line and content
    \hbox{%
      \kern-0.15em%                        % Distance between content and left side of box, negative values for lines shorter than content
      \ifmmode#1\else\ensuremath{#1}\fi%  % The content, typeset in dependence of mode
      \kern-0.15em%                        % Distance between content and left side of box, negative values for lines shorter than content
    }% end of hbox
  }% end of vbox
}

\newcommand{\Cprior}{\mathcal{C}_{\mathup{pr}}}
\newcommand{\Cpost}{\mathcal{C}_{\mathup{post}}}
\newcommand{\ncov}{\vec{\Gamma}_{\!\mathup{noise}}}

\newcommand{\like}{\pi_{\mathup{like}}}
\newcommand{\obs}{\vec{y}}

\newcommand{\ipar}{\theta}

\newcommand{\iparpr}{\ipar_{\mathup{pr}}}
\newcommand{\iparmap}{\ipar_{\scriptscriptstyle\mathup{MAP}}}
\newcommand{\dpar}{\boldsymbol{\theta}}

\newcommand{\dpostmean}{\dpar_\text{post}^\obs}
\newcommand{\dparpr}{\dpar_{\mathup{pr}}}
\newcommand{\priorcov}{\boldsymbol{\Gamma}_{\mathup{prior}} }

\newcommand{\postcov}{\mat{\Gamma}_{\text{post}} }
\newcommand{\priorm}{\mu_{\text{pr}}}
\newcommand{\postm}{\mu_{\text{post}}^{{\obs}}}
\newcommand{\M}{\mat{M}}                       % the Mass matrix
\newcommand{\Ns}{{n_s}}
\newcommand{\Nt}{{n_t}}
\newcommand{\Nd}{{n_y}}
\newcommand{\Nm}{{n}}

\renewcommand{\L}{\mathup{L}}
\newcommand{\Lsym}{\mathup{L}_\mathup{sym}}
\newcommand{\eeta}{\vec{\eta}}

\newcommand{\eip}[2]{\left\langle{#1}, {#2}\right\rangle}

\newcommand{\Hd}{\mat{H}}

\newcommand{\HMd}{\mat{H}_{\mathup{m}}}
\newcommand{\HMpd}{\pmb{\mathcal{H}}_\mathup{m}}
\newcommand{\HMs}{\pmb{\mathbb{H}}_\mathup{m}}

\newcommand{\commentout}[1]{\iffalse {#1} \fi}

\newcommand{\FF}{\mat{F}}
\newcommand{\FFp}{\pmb{\mathcal{F}}}
\newcommand{\logdet}{\log\det}
\newcommand{\DKL}[2]{D_\mathup{KL}\left({#1} \| {#2}\right)}

\newcommand{\DKLa}[2]{\widehat{D}_\mathup{KL}\left({#1} \| {#2}\right)}

\newcommand{\cipfd}[2]{\displaystyle\left\langle{#1},{#2}\right\rangle_{\!\priorcov^{-1}}}
\newcommand{\avep}[1]{\ave_{\priorm}\left\{ {#1} \right\}}
\newcommand{\avey}[1]{\ave_{{\obs|\dpar}}\left\{ {#1} \right\}}

\newcommand{\expect}[2][\boldsymbol{\Omega}]{\mathbb{E}_{#1}\! \left[ #2\right]}
\newcommand{\bmat}[1]{\begin{bmatrix} #1 \end{bmatrix}}
\newcommand{\dd}{\bar{\obs}}
\newcommand{\obsop}{\mathcal{B}}

\newcommand*{\transymb}{\mkern-1.5mu\mathsf{T}}                
\newcommand*{\tran}{^{\transymb}}                % Transpose

\usepackage{enumitem}
\usepackage[]{color}

\title{Efficient D-optimal design of experiments for infinite-dimensional Bayesian linear inverse problems}
\author{
Alen Alexanderian\thanks{Department of Mathematics, North Carolina State University,
Raleigh, NC 27695-8205, USA 
(\texttt{alexanderian@ncsu.edu}, \texttt{http://www4.ncsu.edu/{\char'176}aalexan3/})}
\and	
Arvind K. Saibaba\thanks{Department of Mathematics, North Carolina State University,
Raleigh, NC 27695-8205, USA 
(\texttt{asaibab@ncsu.edu}, \texttt{http://www4.ncsu.edu/{\char'176}asaibab/})}
}

\begin{document}
\maketitle

\begin{abstract} We develop a computational framework for D-optimal
experimental design for PDE-based Bayesian linear inverse problems
with infinite-dimensional parameters.
We follow a formulation of the 
experimental design problem that remains valid in the infinite-dimensional limit. 
The optimal design is obtained by solving
an optimization problem that involves repeated evaluation
of the log-determinant of high-dimensional operators
along with their derivatives. Forming and manipulating these operators is
computationally prohibitive for large-scale problems.  Our methods exploit
the low-rank structure in the inverse problem in three different ways, yielding efficient algorithms.
Our main approach is to use randomized estimators for computing the D-optimal
criterion, its derivative, as well as the Kullback--Leibler divergence from
posterior to prior.  Two other alternatives are proposed based on low-rank
approximation of the prior-preconditioned data misfit Hessian, and a fixed low-rank approximation of the
prior-preconditioned forward operator. Detailed error analysis is provided
for each of the methods, and their effectiveness is demonstrated on a model sensor
placement problem for initial state reconstruction in a time-dependent
advection-diffusion equation in two space dimensions. 

\end{abstract}

% Keywords and AMS subject classifications


%\tableofcontents

%
% Introduction
%
\section{Introduction}
Mathematical models of physical systems enable predictions of future outcomes.
These mathematical models---typically described by a system of partial differential
equations (PDEs)---often contain parameters such as initial conditions, boundary
data, or coefficient functions that are unspecified and have to be estimated.
This estimation requires solving inverse problems where one uses experimental
data and the model to infer unknown model parameters.  Experimental design is
the process of specifying how measurement data, used in parameter inversion,
are
collected.
% and has the two competing objectives of collecting informative data
%at a minimal cost. 
%
On the one hand, the quantity and quality of data can critically impact the
accuracy of parameter reconstruction. On the other hand, acquisition of
experimental data is an expensive process that requires the deployment of
scarce resources.  Hence, it is of great practical importance to optimize the
design of experiments so as to maximize the information gained in the parameter
estimation or reconstruction.  

Optimal experimental design (OED) controls experimental conditions by optimizing certain design criteria
subject to physical or budgetary constraints.  In large-scale applications, 
using commonly used experimental
design criteria requires repeated evaluations of functionals (e.g., trace or
determinant) of high dimensional and expensive-to-apply
operators---applications of these operators to vectors involve expensive PDE
solves. This is a fundamental challenge in OED for large-scale PDE-based
Bayesian inverse problems for which we develop novel algorithms.

In our target applications, the Bayesian inverse problem seeks to infer an
infinite-dimensional parameter using experimental measurements, which are
collected at a set of sensor locations. 
The OED problem aims to find an optimal
placement of sensors that optimize the statistical quality of the 
inferred parameter.
We focus on D-optimal
design of sensor networks at which observational data are collected.  That is,
we seek sensor placements that maximize the expected information gain in
parameter inversion.  Our approach, however, is general in that it is
applicable to D-optimal experimental design for a broad class of linear inverse
problems.

%We focus Bayesian linear inverse problems governed by partial differential
%equations (PDEs) with inversion parameters that are infinite-dimensional.  See
%e.g.,~\cite{Stuart10,DashtiStuart16,Bui-ThanhGhattasMartinEtAl13} for theory
%and numerical methods for such problems. 
%Before solving the
%inverse problem, one has to specify or design appropriate experimental
%conditions for acquiring data. This adds an additional layer of mathematical
%and computational difficulty to an already challenging problem.  

%\paragraph{A motivating application} We explain OED using an example application:
%sensor placement for contaminant source identification (CSI).  The transport of
%contaminants is important to study in the subsurface and urban
%environments~\cite{HemondFechner14}. The inverse problem of interest here is
%using sensor measurements of contaminant concentration to reconstruct the
%source of the contaminant, i.e., estimate the (uncertain) initial state. CSI is
%an example of a linear inverse problem. Only a limited number of sensors can be
%deployed.  Therefore, it is of practical importance to find optimal sensor
%placements.  This can be treated as an OED problem; stated as follows:  of all
%possible sensor placements, find one that minimizes the uncertainty in the
%estimated initial state.

\paragraph{Related Work} OED  is a thriving area of
research~\cite{Ucinski05,AtkinsonDonev92,Pukelsheim93,Pazman86,ChalonerVerdinelli95,
BauerBockKorkelEtAl00,KorkelKostinaBockEtAl04,
HaberHoreshTenorio10,HoreshHaberTenorio10,ChungHaber12,
HuanMarzouk13,LongScavinoTemponeEtAl13,HuanMarzouk14,
AlexanderianPetraStadlerEtAl14, LongMotamedTempone15, Ucinski15,
AlexanderianPetraStadlerEtAl16, AlexanderianGloorGhattas16,
BisettiKimKnioEtAl16, CrestelAlexanderianStadlerEtAl17, YuZavalaAnitescu17,
WalshWildeyJakeman17,ruthotto2017optimal}.  In particular, 
%experimental design with a specific design criterion, namely 
A-optimal experimental design
for large-scale applications has been addressed
in~\cite{HaberHoreshTenorio08,HaberHoreshTenorio10,HoreshHaberTenorio10,HaberMagnantLuceroEtAl12,TenorioLuceroBallEtAl13,
AlexanderianPetraStadlerEtAl14,AlexanderianPetraStadlerEtAl2016,CrestelAlexanderianStadlerEtAl17}.
Computing A-optimal designs for large-scale applications faces similar
challenges to the D-optimal designs.  The use of Monte-Carlo trace estimators~\cite{AvronToledo11}
has been instrumental in making A-optimal design computationally feasible for
such applications.
%problems for which the parameter dimension can be large.  As was mentioned
%previously, just the mere evaluation of the D-optimal criterion is
%computationally challenging.  OED based on other criteria, such as empirical
%Bayes risk, are also available in the literature~\cite{ruthotto2017optimal}.   

Fast algorithms for computing expected information gain, i.e., the D-optimal
criterion, for nonlinear inverse problems, via Laplace approximations, were
proposed in~\cite{LongMotamedTempone15,LongScavinoTemponeEtAl13}. The focus of
the aforementioned works was efficient computation of the D-optimal criterion.
The optimal design was then computed by an exhaustive search over prespecified
sets of experimental scenarios.  We also
mention~\cite{HuanMarzouk13,HuanMarzouk14,huan2016sequential}
that use polynomial chaos (PC) expansions to build easy-to-evaluate surrogates
of the forward model. The expected information gain is then evaluated using an
appropriate Monte Carlo procedure. In~\cite{HuanMarzouk14,huan2016sequential},
the authors use a gradient based approach to obtain the optimal design.
In~\cite{tsilifis2017efficient}, the authors propose an alternate design
criterion given by a lower bound of the expected information gain, and use PC
representation of the forward model to accelerate objective function
evaluations. 
However, the approaches based on PC representations remain limited
in scope to problems with low to moderate parameter dimensions (e.g., parameter
dimensions in order of tens). 

Efficient estimators for the evaluation of D-optimal criterion were developed
in~\cite{saibaba2015fast,SaibabaAlexanderianIpsen17}; however, these works do
not discuss the problem of computing  D-optimal designs.  The mathematical
formulation of Bayesian D-optimality for infinite-dimensional Bayesian linear
inverse problems was established in~\cite{AlexanderianGloorGhattas16}.
A scalable solver for D-optimal designs for infinite-dimensional Bayesian
inverse problems, however, is currently not available to the best of our
knowledge. Here by scalable we mean
methods whose computational complexity, in terms of the number of required PDE
solves, do not scale with the dimension of the discretized inversion
parameter.

%For linear inverse problems, as we show in section~\ref{sec:criterion}, the
%expected information gain and  D-optimal criterion yield equivalent designs.

\paragraph{Contributions} We propose a general computational framework for
D-optimal experimental design for infinite-dimensional Bayesian linear inverse
problems governed by PDEs, with Gaussian prior and noise models.  Specifically, we propose
three approaches for approximately computing the D-optimal criterion,
its gradient, and the Kullback--Leibler (KL) divergence from the 
posterior to prior: 
(i) truncated spectral decomposition approach; (ii) randomized
approach; and (iii) an approach 
based on
fixed low-rank approximation of the prior-preconditioned forward operator.
We compare and
contrast the different approaches and perform a detailed error analysis for 
each case.
The effectiveness of the algorithms is
illustrated on a model problem involving initial state inversion in a
time-dependent advection-diffusion equation.  Besides the computational
framework, we show the relationship between the expected KL divergence and the
D-optimal design criterion; while this is addressed
in~\cite{AlexanderianGloorGhattas16}, in the infinite-dimensional Hilbert space
setting, we present a simple derivation within the context of the discretized
problem, which leads to the D-optimal criterion that is meaningful in the
infinite-dimensional limit.

\paragraph{Paper overview} In section~\ref{sec:background}, we review the
Bayesian formulation of the inverse problem in infinite dimensions, with
special attention to discretization and finite dimensional representation. We
also review our recent work~\cite{SaibabaAlexanderianIpsen17} on randomized
estimators for determinants. In section~\ref{sec:criterion}, we provide a
simple derivation of the Bayesian D-optimal criterion for infinite-dimensional
problems and present the formulation of the sensor placement problem as a
D-optimal design problem.  Section~\ref{sec:algorithms} is devoted to
developing efficient algorithms for computing the KL divergence from posterior to prior, and 
for evaluating the D-optimal objective function and its gradient; also, 
in that section, we formulate the optimization problem for finding D-optimal
sensor placements.
%
%the derivation of our proposed algorithms for fast computation of the
%KL-divergence, and our algorithms for efficient computation of the D-optimal
%criterion and its derivative with respect to design parameters. 
%
In section~\ref{sec:model}, we outline our model sensor placement problem for
initial state inversion in a time-dependent advection-diffusion equation.
Section~\ref{sec:numerics} contains our numerical results, where we illustrate
effectiveness of our methods. Finally, in section~\ref{sec:conclusions}, we
provide some concluding remarks.

%
% Background material
%
\section{Background}\label{sec:background}
In this section, we provide the requisite background material required for the
formulation and numerical computation of optimal experimental designs for
infinite-dimensional Bayesian inverse problems. 

\subsection{Bayesian linear inverse problems}
Here, we describe the main ingredients of a Bayesian linear inverse problem,
and our assumptions about them.  The inference parameter is denoted by 
$\ipar$ and is an element of an infinite-dimensional real separable Hilbert space
$\hilb$.
We consider a Gaussian measure
$\priorm=\GM{\iparpr}{\Cprior}$ as the prior distribution law of $\ipar$. 
The prior mean $\iparpr$ is assumed to be a  sufficiently
regular element of $\hilb$, and $\Cprior:\hilb \to \hilb$ a
strictly positive self-adjoint trace-class operator.  Following
\cite{Bui-ThanhGhattasMartinEtAl13,Stuart10,DashtiStuart16}, we consider prior
covariance operators of the form $\C = \A^{-2}$, where $\A$ is a Laplacian-like
operator (in the sense defined in~\cite{Stuart10}).  This ensures that, in two
and three space dimensions, $\C$ is trace-class. 

We consider observations taken at $\Ns$ measurement points, which henceforth we refer to
as sensor locations, and at $\Nt$ discrete points in time. We denote the
vector of experimental (sensor) data by $\obs \in \R^\Nd$, where  $\Nd =
\Ns\Nt$. We assume an  additive Gaussian noise model,  
\begin{equation*}
\obs = F\ipar + \vec{\eta}, \quad \eeta \sim \GM{\vec{0}}{\ncov}.
\end{equation*} 
Here $\ncov\in \R^{\Nd\times \Nd}$ is the noise covariance
matrix, and $F:\hilb \to \R^\Nd$ is a linear parameter-to-observable map, whose evaluation
involves solving the governing PDEs, and applying an observation operator that
extracts solution values at the sensor locations and at observation times. Due
to our choice of the noise model, the likelihood probability density function
(pdf), $\like(\obs | \ipar)$, is given by  
\begin{equation*}
\label{equ:likelihood} \like(\obs \mid \ipar) \propto \exp\left\{ -\frac12
\big(F\ipar - \obs\big)\tran \ncov^{-1} \big(F\ipar - \obs\big)\right\}.
\end{equation*} 
The solution of a Bayesian inverse problem is the posterior
measure, $\postm$, which describes the distribution law of the inference
parameter $\ipar$ conditioned on the observed data.  The Bayes Theorem describes
the relation between the prior measure, the data likelihood, and the posterior
measure. In the infinite-dimensional Hilbert space settings, the Bayes Formula is
given by \cite{Stuart10}, 
\[    \frac{d\postm}{d\priorm} \propto \like(\obs \mid \ipar). \] 

Here, $\frac{d\postm}{d\priorm}$ is the Radon-Nikodym
derivative~\cite{Williams1991} of  $\postm$ with respect to the $\priorm$. In
the present work, the parameter-to-observable map is linear and is assumed to be 
continuous.
Under these assumptions on the parameter-to-observable map, the noise model, and the prior, it
is known~\cite[Section 6.4]{Stuart10} that the solution $\postm$ of the
Bayesian \emph{linear} inverse problem is a Gaussian measure $\postm =
\GM{\iparmap}{\Cpost}$ whose mean and covariance operator are given by
\begin{subequations}\label{equ:mean-cov} 
\begin{align}
\iparmap &=  \Cpost(F^*\ncov^{-1}\obs + \Cprior^{-1}\iparpr), \\
\Cpost &=  (F^* \ncov^{-1} F + \Cprior^{-1})^{-1}. 
\end{align}  
\end{subequations}
Note that in this case the posterior mean coincides with the MAP estimator.

\subsection{Discretization of the Bayesian inverse problems}

Following the setup
in~\cite{Bui-ThanhGhattasMartinEtAl13,AlexanderianPetraStadlerEtAl14}, we state
the finite-dimensional Bayesian inverse problem so that  it is consistent with
the underlying inference problem with $\hilb = L^2(\D)$ as its parameter space.  

\paragraph{Finite-dimensional Hilbert space setup}
We use a finite-element discretization of the Bayesian inverse problem, 
and hence work in the $n$-dimensional inner product space 
$\Rnm$, which is  
$\R^\Nm$ equipped the mass-weighted inner product, $\mip{\vec{x}}{\vec{y}} =
\ip{\M\vec{x}}{\vec{y}}$; here $\ip{\cdot}{\cdot}$ denotes the Euclidean inner
product, $\M$ is the finite-element mass
matrix, and $\Nm$ is the dimension of the discretized inversion parameter;
see~\cite{Bui-ThanhGhattasMartinEtAl13,AlexanderianPetraStadlerEtAl14}.

We denote by $\L(\Rnm)$, and $\L(\R^k)$ the spaces of linear operators on $\Rnm$, and $\R^k$, respectively. We also
define the spaces of the linear transformations from $\Rnm$ to $\R^k$, and from
$\R^k$ to $\Rnm$ by $\L(\Rnm,\R^k)$, and $\L(\R^k, \Rnm)$, respectively. Here
$\R^k$ is assumed to be equipped with the Euclidean inner product. 
%
%Let $\mathcal{X} = (\R^n,\mip{\cdot\,}{\cdot})$ and $\mathcal{Y} = (\R^k,\ip{\cdot\,}{\cdot})$  be 
%two finite-dimensional Hilbert spaces. Furthermore, let $\L(\mathcal{X},\mathcal{Y})$ be the space of linear operators 
%from $\mathcal{X}$ to $\mathcal{Y}$. We drop the second argument if the linear operator maps $\mathcal{X}$ to itself.
%
Next, we give explicit formulas for the adjoints of elements of these spaces, which
will be useful in the rest of the paper: 
\begin{align}
   \text{for } \mat{A} \in \L(\Rnm),       & \quad \mat{A}^* = \M^{-1} \mat{A}\tran \M, \\
   \text{for } \mat{A} \in \L(\Rnm,\R^k), &\quad \mat{A}^* = \M^{-1} \mat{A}\tran, \text{ and}\\
   \text{for }  \mat{A} \in \L(\R^k,\Rnm), &\quad \mat{A}^* = \mat{A}\tran \M.
\end{align}
The superscript ${}\tran$ denotes matrix transpose. 
Also, we use the notations $\Lsym(\R^n)$ and $\Lsym(\Rnm)$ 
for the subspaces of self-adjoint operators in $\L(\R^n)$ and $\L(\Rnm)$, respectively. 
%

%%%%%
\iffalse
In what follows, we use the convenient notation 
notation $\R^n_{\M}$ for the finite-dimensional Hilbert space 
$(\R^n, \mip{\cdot\,}{\cdot})$.   We define the following spaces of linear transformations:
\[
\begin{aligned}
\L(\R^n)       &= \text{space of linear operators on $\R^n$ endowed with the Euclidean inner product}, 
\\
\L(\Rnm)       &= \text{space of linear mappings } \mat{A}:(\R^n, \mip{\cdot\,}{\cdot}) \to (\R^n, \mip{\cdot\,}{\cdot}),
\\
\L(\Rnm, \R^\Nd) &= \text{space of linear mappings } \mat{A}:(\R^n, \mip{\cdot\,}{\cdot}) \to (\R^\Nd, \ip{\cdot\,}{\cdot}),
\\
\L(\R^k,\Rnm)  &= \text{space of linear mappings } \mat{A}:(\R^k, \ip{\cdot\,}{\cdot}) \to (\R^n, \mip{\cdot\,}{\cdot}). 
\end{aligned}
\]
The adjoint of a linear mapping $\mat{A} \in \L(\Rnm)$ 
is given by~\cite{Bui-ThanhGhattasMartinEtAl13}, $\mat{A}^* = \M^{-1} \mat{A}\tran \M$; moreover,
and
\begin{align}
   \text{for } \mat{A} \in \L(\Rnm,\R^\Nd), &\quad \mat{A}^* = \M^{-1} \mat{A}\tran,
\\
   \text{for } \mat{A} \in \L(\R^k,\Rnm), &\quad \mat{A}^* = \mat{A}\tran \M.
\end{align}
Here the superscript ${}\tran$ is the usual matrix transpose.  Moreover, for a linear
operator $\mat{A} \in \L(\R^n)$, we denote its adjoint by the traditional
notation $\mat{A}\tran$.  
\fi
%%%%

\begin{remark} \label{r_sym}
It is straightforward to note that for $\mat{A} \in \Lsym(\Rnm)$, 
$$\mat{B} = \M^{1/2} \mat{A} \M^{-1/2} \in \Lsym(\R^n).$$  
Thus, the  similarity transform $\mat{A} \mapsto \mat{M}^{1/2} \mat{A} \mat{M}^{-1/2}$ 
gives a matrix that is self-adjoint with respect to the Euclidean inner product; this is useful 
in practical computations.
\end{remark}

\paragraph{The discretized inverse problem} 
We denote by $\dpar \in \Rnm$ the discretized parameter that we wish to infer, and as before, let
$\obs \in \R^\Nd$ 
be a vector containing sensor measurements, i.e., the experimental data. 
As discussed in~\cite{Bui-ThanhGhattasMartinEtAl13},
the discretized inverse problem is stated on the space $\Rnm$. The prior 
density is given by
\begin{equation}
\label{eq:prior_pdf}
\pi_\text{prior}(\dpar) \:\propto\:
 \exp\left\{
- \frac 12 \left\| \mat{A} (\dpar - \dparpr) \right\|^2_{\M}
\right\}. 
\end{equation}
Following~\cite{Bui-ThanhGhattasMartinEtAl13}, we define $\mat{A} = \M^{-1}\mat{L}$, with 
\begin{equation}\label{equ:prior_sqrt}
   \mat{L} = \alpha \mat{K} + \beta \mat{M},
\end{equation}
where $\mat{K}$ is the finite-element stiffness matrix, and $\alpha$ and
$\beta$ are positive constants.
Note that $\mat{A} \in \Lsym(\Rnm)$ and $\priorcov = \mat{A}^{-2}$.
The posterior measure is also Gaussian with mean and covariance operator 
given by~\cite{Bui-ThanhGhattasMartinEtAl13,Stuart10},
\begin{subequations}\label{equ:discrete-mean-cov}
   \begin{align}
   \dpostmean &=  \postcov \left(\FF^*\ncov^{-1} \obs + \priorcov^{-1} \dparpr\right),\label{equ:discrete-mean} \\
   \postcov &= \left(\FF^*\ncov^{-1}\FF + \priorcov^{-1} \right)^{-1}, \label{equ:discrete-cov}
   \end{align}
\end{subequations}
where 
$\FF \in \L(\Rnm, \R^\Nd)$ is the discretized 
parameter-to-observable map (forward operator). Note that~\eqref{equ:discrete-mean-cov} 
is the discretized version of~\eqref{equ:mean-cov}.
%
%
%Suppose we have a linear parameter-to-observable map $\FF$ and suppose the additive model for
%data,
%\[
%    \obs = \FF \dpar + \vec{\eta},
%\]
%where $\vec{\eta}$ is a centered Gaussian with covariance $\ncov$, so that
%\[
%   \obs \sim \GM{\FF\dpar}{\ncov}.
%\]
%We use a Gaussian prior measure for $\dpar$, which with no loss of generality we asssume to be centered,
%and denote the prior covariance by $\priorcov$.
%
As is well known, 
the posterior mean $\dpostmean$ is the unique global minimizer of
the regularized data misfit functional 
\[
\mathcal{J}(\dpar) :=  
\frac12 \ip{\FF\dpar - \obs}{\ncov^{-1} (\FF\dpar - \obs)} + \frac12 
\mip{\dpar - \dparpr}{\priorcov^{-1}(\dpar-\dparpr)}.
\]

%given by the sum of the negative log-likelihood, and 
%regularization term weighted by prior covariance operator; see e.g.,~\cite{Tarantola05,
%Bui-ThanhGhattasMartinEtAl13,AlexanderianPetraStadlerEtAl14}. 
%\begin{equation}\label{equ:linearinv}   \mathcal{J}(\dpar) :=  \frac12 \ip{\FF
%\dpar - \obs}{\ncov^{-1} (\FF\dpar - \obs)} + \frac12 \mip{\dpar -
%\dparpr}{\priorcov^{-1} (\dpar-\dparpr)}. \end{equation}
%\begin{equation}\label{equ:dparh} %    \dpostmean =
%\Hd^{-1}\FF^*\ncov^{-1}\obs, %\end{equation} 
The Hessian operator of this
functional  is $\Hd = \HMd + \priorcov^{-1}$, where $\HMd = \FF^*
\ncov^{-1} \FF$ is the data misfit Hessian~\cite{Tarantola05,Bui-ThanhGhattasMartinEtAl13}.  
Note also that
\[
\postcov = (\HMd + \priorcov^{-1})^{-1} = 
\priorcov^{1/2} (\priorcov^{1/2} \HMd \priorcov^{1/2} + \mat{I})^{-1} \priorcov^{1/2}.  
\] 
The operator 
\[
\HMpd = \priorcov^{1/2} \HMd \priorcov^{1/2} \in \Lsym(\Rnm)
\] 
is referred to as the prior-preconditioned data misfit Hessian~\cite{Bui-ThanhGhattasMartinEtAl13} and will 
be of relevance in the sequel. 
%This operator belongs to $\Lsym(\Rnm)$. 
Following Remark~\ref{r_sym}, 
we also define  
\begin{equation}\label{equ:HMs}
%\HMs = \mat{M}^{1/2} \priorcov^{1/2} \HMd \priorcov^{1/2} \mat{M}^{-1/2},
\HMs = \mat{M}^{1/2} \HMpd \mat{M}^{-1/2},
\end{equation}
which is a symmetric matrix and is convenient to use in computations.
The rank of $\HMs$  is $\min\{\Nd,n\}$. For some
applications, even though the rank can be large, the eigenvalues of $\HMs$
decay rapidly.  This is further discussed in~\cite{FlathWilcoxAkcelikEtAl11}. 
%instead of $\priorcov^{1/2} \HMd \priorcov^{1/2}$. 
We also define the prior-preconditioned forward operator 
\begin{equation}\label{eqn:ffp}
\FFp = \FF \priorcov^{1/2}.
\end{equation} 

\begin{remark}
Preconditioning the forward operator 
by the prior covariance operator is expected to lead to faster 
decay of singular values of the preconditioned operator, when using 
a smoothing $\priorcov$. 
This can be explained using the
multiplicative singular value inequalities~\cite[page~188]{Bhatia09},
\[
   \sigma_i(\FF \priorcov^{1/2}) \leq
   \sigma_1(\priorcov^{1/2})\sigma_i(\FF), 
   \quad \text{and} \quad 
   \sigma_i(\FF \priorcov^{1/2}) \leq \sigma_i(\priorcov^{1/2}) \sigma_1(\FF).
\]
The first inequality shows that preconditioning by $\priorcov^{1/2}$ does not
degrade the decay of singular values of $\FF$, and the second one shows that
a smoothing $\priorcov$ for which singular values $\sigma_i(\priorcov^{1/2})$ decay faster than that of 
$\FF$ can improve the spectral decay of $\FF$. The latter was observed numerically 
in many cases; see e.g.~\cite{FlathWilcoxAkcelikEtAl11,AlexanderianPetraStadlerEtAl14}.
\end{remark}

\begin{remark} The expression for $\HMs$ given above involves square root of
$\mat{M}$. While this can be computed for some problems, 
a numerically viable alternative is to
compute a Cholesky factorization $\mat{M} = \mat{L}\mat{L}\tran$, 
where $\mat{L}$ is lower triangular with positive diagonal entries.  
Then, it is
simple to show that the transformation $\mat{A} \mapsto \mat{L}\tran \mat{A}
\mat{L}^{-\transymb}$, maps elements of $\Lsym(\Rnm)$ to $\Lsym(\R^n)$. This
alternate formulation can be used to redefine $\HMs$ as an element of
$\Lsym(\R^n)$. Computing explicit factorizations may be prohibitively expensive
for large-scale problems in three-dimensional physical domains. In this case, the action of
$\mat{M}^{1/2}$ (and its inverse) on vectors can be computed using polynomial
approaches~\cite{chen2011computing,chow2014preconditioned} or rational
approaches~\cite{hale2008computing,bakhos2017multipreconditioned}.

%If these approaches are too computationally cumbersome, another
%approach, suitable for large-scale problems, involves approximating
%applications of $\mat{M}^{1/2}$ (and its inverse) on vectors using approaches
%such ...~\alennote{Arvind, can you complete this?}.  
\end{remark}

\subsection{Randomized log-determinant estimators}\label{ss_logdet} Here we
describe the randomized log-determinant estimator developed in our previous
work~\cite{SaibabaAlexanderianIpsen17}. These estimators will be used for 
efficient estimation of $\logdet (\mat{I} + \HMs)$. 
%In the sequel, we will develop and use randomized estimators that exploit the
%low-rank structure of $\HMs$. 
%Recall that $\HMs$ is self-adjoint with respect
%to the mass-weighted inner product; following Remark~\ref{r_sym}
%$\mat{M}^{1/2}\HMs\mat{M}^{-1/2}$ is symmetric with respect to the Euclidean
%inner product and %we denote it by $\HMs$ and  
%work with it instead. 
Recall that $\HMs$
is symmetric positive semidefinite and has rapidly decaying 
eigenvalues---a problem structure that is exploited by the randomized algorithms
discussed here.
Moreover, $\HMs$  is a dense operator that is too computationally expensive
to build elementwise, but matrix-vector products involving $\HMs$ can be
efficiently computed. To compute the log-determinant of $\mat{I} + \HMs$, we
first use randomized subspace iteration to estimate the dominant subspace.
Algorithm~\ref{alg:randsvd} is an idealized version of randomized subspace
iteration used for this purpose. As the starting guess, we sample random matrix
$\mat{\Omega}$ with $\ell\geq k$ columns, with i.i.d.\ entries from the
standard normal distribution. In practice, $\ell = k + p$, with $p$ a small 
(e.g., $p = 10$) oversampling parameter. 
We then apply  $q$ steps of the power iterations
to this random matrix using $\HMs$ to obtain $\mat{Y}$. A thin QR decomposition
of $\mat{Y}$ produces a matrix $\mat{Q}$ with orthonormal columns. The  output
of Algorithm~\ref{alg:randsvd} is the $\ell\times \ell$ restriction of $\HMs$
to $\text{Span}(\mat{Q})$, i.e., $\mat{T} = \mat{Q}\tran\HMs\mat{Q}$. In
practice, the number of power iterations is taken to be $q = 1$, or $q = 2$.
Specifically, for all the numerical tests in the present work, $q = 1$ was found
to be sufficient.

\begin{remark} Throughout this paper, we assume that the random starting guess
$\mat{\Omega}$
is a standard Gaussian random matrix. An extension to other random starting
guesses, such as Rademacher, can be readily derived, following the approach
in~\cite{SaibabaAlexanderianIpsen17}.  
\end{remark}

\begin{algorithm}[!ht]
\begin{algorithmic}[1]
\REQUIRE Hermitian positive semi-definite matrix $\HMs \in\mathbb{R}^{n\times n}$ with
target rank $k$,\\
$\qquad$ number of subspace iterations $q\geq 1$,\\
$\qquad$ starting guess 
$\mat{\Omega} \in \mathbb{R}^{n\times \ell}$ with $k\leq \ell \leq n$ columns.
\ENSURE Matrix $\mat{T}\in\mathbb{R}^{\ell\times \ell}$.
\STATE Multiply $\mat{Y} = (\HMs)^q\mat{\Omega}$.
\STATE  Thin QR factorization $\mat{Y}=\mat{Q}\mat{R}$.
\STATE Compute $\mat{T} = \mat{Q}\tran\HMs\mat{Q}$.
\end{algorithmic}
\caption{Randomized subspace iteration (idealized version)}
\label{alg:randsvd}
\end{algorithm}

With the output of Algorithm~\ref{alg:randsvd}, we can approximate the
objective function as 
\[    
\logdet(\mat{I} + \HMs) \approx
\logdet(\mat{I}+\mat{T}). 
\] 
We also mention that this algorithm can be used to compute an (approximate) 
low-rank approximation of $\HMs$ as follws:
%In section~\ref{ss_obj_grad} it will be necessary to compute low-rank
%approximation to $\HMs$. This can be done as follows: 
compute the eigenvalue decomposition of
$\mat{T} = \mat{U}_T\mat{\Lambda}_T\mat{U}_T\tran$, and compute 
$\widehat{\mat{U}} =
\mat{QU}_T$. Then we have the low-rank approximation 
\[ 
   \HMs \approx \widehat{\mat{U}} \widehat{\mat{\Lambda}} \widehat{\mat{U}}\tran, \quad \text{with } 
   \widehat{\mat{\Lambda}} = \mat{\Lambda}_T.
\]

For the error analysis, we will need the following definitions. 
Since $\HMs$ is symmetric positive semidefinite, its eigenvalues are nonnegative 
and can be arranged in descending order. Consider the eigendecomposition
$\HMs = \mat{U}\mat{\Lambda}\mat{U}\tran$; let us suppose that  $\mat{\Lambda}_1 \in
\mathbb{R}^{k\times k}$ contains the dominant eigenvalues and $\mat{\Lambda}_2
\in \mathbb{R}^{(n-k)\times (n-k)}$, where $k$ is the target rank. Partition the
eigendecomposition as 
\begin{equation}\label{e_evd}  \HMs =
\bmat{\mat{U}_1 & \mat{U}_2} \bmat{\mat{\Lambda}_1 \\ &
\mat{\Lambda_2}}\bmat{\mat{U}_1\tran \\ \mat{U}_2\tran}. 
\end{equation} 
We assume that $\mat\Lambda_1$ is nonsingular, and that there is a gap between
the eigenvalues $\lambda_k$ and $\lambda_{k+1}$. The size of the gap is
inversely proportional to $$\gamma\equiv \lambda_{k+1}/\lambda_k
=\|\mat{\Lambda}_2\|_2\,\|\mat{\Lambda}_1^{-1}\|_2<1.$$
As noted before, the rank of $\HMs$, which we denote by $K$, satisfies
$K =  \min\{\Nd,n\}$. 
%\begin{equation} \label{eqn:rank}
%K =  \min\{\Nd,n\}, 
%\end{equation}
%where as before $\Nd = \Ns\Nt$. 

%
% D-optimal design of experiments: basic definitions/theory
%
\section{D-optimal design of experiments for PDE based inverse problems}
\label{sec:criterion}
In this section we examine the information theoretic connections
of the D-optimality criterion, and formulate the sensor placement problem as
an OED problem. 

%To this end, we briefly review the Kullback-Liebler divergence and show that
%the D-optimal criterion is nothing but the expected Kullback-Liebler
%divergence draw connections between the Kullback-Liebler 

%Before discussing the D-optimal criterion, we briefly review the
%Kullback-Liebler (KL) divergence and discuss its limit in the infinite
%dimensional setting. We then establish the connection between the KL divergence
%and D-optimal criterion. We also explain how to formulate the sensor placement
%problem as an OED problem. 

\subsection{Kullback-Leibler divergence} The Kullback-Liebler (KL) divergence~\cite{KullbackLeibler51}
is a ``measure'' of how one probability distribution deviates from another.
It is important to note that it is not a true metric on the set of
probability measures, since it is not symmetric and does not satisfy the
triangle inequality~\cite{sullivan2015introduction}. Despite these shortcomings, the KL divergence is widely
used since it has many favorable properties. In the context of Bayesian
inference, the KL divergence measures the information gain between the prior
and the posterior distributions. 
We motivate the discussion by
recalling the form of the KL divergence from posterior to prior
distributions in the finite-dimensional case.  To avoid introducing new notation, we continue to denote the
(finite-dimensional) posterior and prior measures by 
\begin{equation}\label{eqn:finitemeasure}
\priorm = \GM{\vec{0}}{\priorcov}, \qquad \postm = \GM{\dpostmean}{\postcov},
\end{equation} respectively.  
Here, for
convenience, and without loss of generality, we assumed the prior mean is
zero.  The following expression for $\DKL{\postm}{\priorm}$ is well
known~\cite[Exercise 5.2]{sullivan2015introduction}. 
\begin{multline}\label{equ:DKL-fd} \DKL{\postm}{\priorm} =
\frac{1}{2}\,\left[ -\log\left( \frac{\det\postcov}{\det \priorcov}\right) - n
+ \trace(\priorcov^{-1} \postcov) \right. \\ \left. +
\cipfd{\dpostmean}{\dpostmean} \right], \end{multline} where we have used the
notation, \[ \cipfd{\vec{x}}{\vec{y}} =
\mip{\priorcov^{-1/2}\vec{x}}{\priorcov^{-1/2}\vec{y}}, \quad \vec{x}, \vec{y}
\in \Rnm.  \] 
We note that the expression~\eqref{equ:DKL-fd} is not meaningful
in the infinite-dimensional case, and it is not
immediately clear what the infinite-dimensional limit will be.
In~\cite{AlexanderianGloorGhattas16}, this issue is investigated in detail in
the infinite-dimensional Hilbert space setting. 
%Following the discussions in~\cite{AlexanderianGloorGhattas16}, 
Here we provide a derivation of an expression for the KL divergence, within the context of 
the discretized problem, that
remains meaningful in infinite dimensions.
The $\logdet$ term in~\eqref{equ:DKL-fd}
may be simplified as: 
\begin{align}\label{equ:detterm} 
-\log\left(\frac{\det\postcov}{\det \priorcov}\right)   
% \log\left(
%\frac{\det\priorcov}{\det \postcov}\right) = 
% \log \det \left( \priorcov \postcov^{-1}\right) \nonumber \\ =& \>
&= \log \det \left( \priorcov^{1/2} (\HMd + \priorcov^{-1}) \priorcov^{1/2}
\right) \\ 
&= \log \det (\HMpd + \mat{I}). \nonumber 
\end{align} Since
$\HMpd$ is the discretization of a trace-class operator, 
$\det (\HMpd + \mat{I})$ is the Fredholm determinant, and is well-defined.  Next, we consider
the term $-n+\trace(\priorcov^{-1} \postcov)$. Since $n = \trace(\mat{I})$ and
the trace operator is linear, 
\begin{align} \nonumber -n+\trace(\priorcov^{-1}
\postcov) = & \> \trace\big( (\priorcov^{-1} - \postcov^{-1})\postcov\big) \\
\label{equ:traceterm} = & \> -\trace(\HMd \postcov) = -\trace(\HMpd(\mat{I} +
\HMpd)^{-1}), 
\end{align}
%
%\begin{eqnarray}\label{equ:traceterm} &&-n+\trace(\priorcov^{-1} \postcov) =
%-\trace(\mat{I}) + \trace(\priorcov^{-1} \postcov) \nonumber \\
%&&\hspace{.25in}= \trace( \priorcov^{-1} \postcov - \mat{I}) = \trace\big(
%(\priorcov^{-1} - \postcov^{-1})\postcov\big) \\   &&\hspace{1.25in}=
%-\trace(\HMd \postcov) = -\trace(\HMpd(\mat{I} + \HMpd)^{-1}), \end{eqnarray}
%
where in the penultimate step we used the relation $\postcov^{-1} = \HMd +
\priorcov^{-1}$.  Notice that the argument of the trace in the final expression
is in fact a trace-class operator in the infinite-dimensional case and has a
well-defined trace.  Combining~\eqref{equ:detterm} and~\eqref{equ:traceterm} we
rewrite~\eqref{equ:DKL-fd} as follows: 
\begin{equation}\label{equ:DKL-fd-alt}
\DKL{\postm}{\priorm} =  \frac{1}{2}\Big[ \log\det (\HMpd + \mat{I}) -
   \trace(\HMpd(\mat{I} + \HMpd)^{-1}) + \cipfd{\dpostmean}{\dpostmean} \Big].
\end{equation} 
While this equation is equivalent to~\eqref{equ:DKL-fd} in
finite dimensions, it will be used henceforth, because it is well defined in
the infinite dimensional limit.  

\subsection{Expected information gain and the D-optimal criterion} 

%The expected information gain is also known as the Shannon entropy. 

The following result provides an explicit expression for the expected
information gain leading to the D-optimal criterion.
\begin{theorem}\label{thm:doptimal_criterion}
Let $\priorm = \GM{\vec{0}}{\priorcov}$ and $\postm = \GM{\dpostmean}{\postcov}$ 
be the prior and posterior measures corresponding to a Bayesian
linear inverse problem with additive Gaussian noise model. Then,
\begin{equation}\label{equ:EKLD}
\avep{\avey{\DKL{\postm}{\priorm}}} = \displaystyle\frac12 \log\det (\HMpd + \mat{I}). 
\end{equation}
\end{theorem}
\begin{proof} See Appendix~\ref{ssec:doptproof}.
\end{proof}

Several important observations are worth mentioning here. First, 
this theorem provides the expression for 
the D-optimal criterion in terms of 
the expected KL divergence from the posterior to the prior
distribution; this 
expression remains meaningful in the infinite-dimensional setting. 
Moreover,~\eqref{equ:EKLD} reveals an important problem structure in the D-optimal
criterion: the
operator $\HMpd$ admits a low-rank representation in a wide class of inverse
problems. This low-rank structure is exploited in our algorithms for numerical
solution of the D-optimal experimental design problem.
Finally, the KL divergence~\eqref{equ:DKL-fd-alt} depends on the data, through the term $\dpostmean$.
For the linear-Gaussian problem, which we consider in this paper, the
D-optimality criterion is independent of the data since the expectation of the
KL divergence is taken over the data and the prior. 

In the context of this paper, an optimal experimental design is one that
maximizes the D-optimal criterion, or alternatively, maximizes the expected
information gain. Moreover, in practical computations, we can use $\HMs$ in replace of
$\HMpd$ in the D-optimal criterion~\eqref{equ:EKLD} and the expression for the KL
divergence~\eqref{equ:DKL-fd-alt}.

\subsection{The D-optimal criterion for sensor placement problems}

For the sensor placement problem, we consider the following
setup~\cite{Ucinski05, HaberHoreshTenorio08,HaberMagnantLuceroEtAl12,
AlexanderianPetraStadlerEtAl14, AlexanderianPetraStadlerEtAl16}. A set of
points $\vec{x}_i$, $i = 1, \ldots, \Ns$, representing the candidate sensor
locations, is identified. Each sensor location is associated with a
non-negative weight $w_i \in \R$ encoding its relative importance.
The experimental design is then fully specified by the vector
$\vec{w} = [w_1, w_2, \ldots, w_\Ns]\tran$.

For sensor placemement, binary weight vectors are desired.  If $w_i$ equals
one, then a sensor will be placed at $\vec{x}_i$; otherwise, no sensors will be
placed at that location.  In this context, nonbinary weights are difficult to
interpret and implement. From a computational standpoint, however, the
experimental design problem with binary weights is a combinatorial problem and
intractable even for a modest number of sensors. Therefore, as
in~\cite{AlexanderianPetraStadlerEtAl14,AlexanderianPetraStadlerEtAl16}, we
consider a relaxation of the problem with weights $w_i \in [0,1]$, $i  = 1,
\ldots, \Ns$. Subsequently, binary weights are obtained using sparsifying
penalty functions; see section~\ref{ssec:dopt}. 
%and using a continuation approach to ensure binary weights. This is discussed
%further in section~\ref{ssec:dopt}. 

We next describe how the design $\vec{w}$ is incorporated in the 
D-optimal criterion. 
The weight vector $\vec{w}$ enters the Bayesian inverse problem through the data likelihood.
Subsequently, following the setup in~\cite{AlexanderianPetraStadlerEtAl14}, 
the data misfit Hessian takes the form,
$\HMd = \FF^* \mat{W}^{1/2}\ncov^{-1}\mat{W}^{1/2} \FF$, where 
$\mat{W}$ is an $\Ns\Nt \times \Ns\Nt$ diagonal matrix given by
\begin{equation}\label{equ:W} \mat{W} = \sum_{j = 1}^\Ns w_j \mat{E}_j, \qquad
	\mat{E}_j = \mat{I}_\Nt \otimes \vec{e}_j \vec{e}_j^T, 
\end{equation}
with $\vec{e}_j$ the $j$th unit vector in $\R^\Ns$.
We consider the case of uncorrelated observations\footnote{This assumption is not crucial to our 
formulation and we can extend it to the case when the data are correlated. }, and hence, $\ncov$ is a
diagonal matrix, $\ncov = \diag(\sigma^2_1, \sigma^2_2, \ldots, \sigma^2_\Ns)$,
where $\sigma^2_j$ indicates the noise level at each sensor. For convenience,
we introduce the notation, $\mat{W}^\sigma = \sum_{j = 1}^\Ns w_j
\mat{E}^\sigma_j, \quad \mat{E}_j^\sigma =   \sigma_j^{-2} \mat{E}_j$. 
In this case, we have, $\HMd = \HMd(\vec{w}) = \FF^*
\mat{W}^\sigma\FF$.  

With the above definitions in place, 
the D-optimal objective becomes 
\begin{equation}\label{equ:Jw}
J(\vec{w}) = \log\det ( \mat{I} + \HMs(\vec{w})), 
\end{equation}
with $\HMs(\vec{w})$ defined according to~\eqref{equ:HMs},
$\HMs(\vec{w}) = \mat{M}^{1/2} \priorcov^{1/2} \HMd(\vec{w}) \priorcov^{1/2} \mat{M}^{-1/2}$.
Using the expression for derivative of 
log-determinant of matrix-valued 
functions~\cite[Theorem B.17]{Ucinski05},
and letting
$\partial_j$ be shorthand for $\frac{\partial}{\partial w_j}$, the derivative of the above function is given by,
\begin{equation}\label{e_partial_J2} 
	\partial_j J(\vec{w}) = \trace( (\mat{I} +
	\HMs(\vec{w}))^{-1} \mat{Z}_j),  
\end{equation} 
where the matrix $\mat{Z}_j$ is 
\begin{equation}
\label{eqn:Zj} \mat{Z}_j \equiv \partial_j \HMs(\vec{w}) =  \mat{M}^{1/2}\FFp^* \mat{E}^\sigma_j \FFp \mat{M}^{-1/2}. 
\end{equation} 
The matrix $\FFp$ was defined in~\eqref{eqn:ffp}. 
The last expression follows since $\partial_j\mat{W}^\sigma = \mat{E}_j^\sigma$.
% and by the linearity of the derivative operator.
%Here we use  
%Note that \[ \HMpd(\vec{w}) =  \priorcov^{1/2} \HMd(\vec{w}) \priorcov^{1/2} =
%\priorcov^{1/2} \FF^* \mat{W}^\sigma \FF \priorcov^{1/2}, \]
 
%where, for convenience, we denote 
%\begin{equation}\label{eqn:Zj}
%	\mat{Z}_j\equiv \mat{M}^{1/2}\FFp^* \mat{E}^\sigma_j \FFp \mat{M}^{-1/2} \qquad j =1,\dots,\Ns
%\end{equation}
%  and, ultimately,  
%%Therefore, by~\eqref{e_partial_J},
%\begin{equation}\label{e_partial_J2} \partial_j J(\vec{w}) = \trace\left(
%(\mat{I} + \HMpd(\vec{w}))^{-1}\mat{Z}_j\right).
%\end{equation}

%
% Numerical algorithms
%
\section{Numerical algorithms}\label{sec:algorithms}
In this section, we present our proposed numerical algorithms for efficient
computation of KL divergence (section~\ref{ss_klde}), and
the D-optimal criterion and its derivative (section~\ref{ss_obj_grad_all}).
The optimization problem for D-optimal sensor placements is formulated in
section~\ref{ssec:dopt}. 
%Our methods exploit the low-rank structure of $\HMs$ either through low-rank 
%spectral decompositions or using efficient randomized methods. 
 
%The matrix $\HMs$ is of rank $K = \min\{\Ns\Nt,n\}$ is known to have rapidly
%decaying eigenvalues. In many applications of interest, it can be approximated
%by a low-rank matrix with rank $k \ll K$. We exploit this property of $\HMs$,
%to estimate the KL divergence, the D-optimal objective function and
%its derivative.  

\subsection{Fast estimation of KL divergence}\label{ss_klde}
To estimate the KL-divergence, we propose two methods---based on the the
truncated spectral representation of $\HMs$ or using a randomized estimator.

\paragraph{Truncated spectral decomposition} We can estimate the KL
divergence,  using only the first $k$ exact eigenvalues of $\HMs$. 
Let $\mat{\Lambda}_1$ and $\mat{\Lambda_2}$ be  
as in~\eqref{e_evd}.
We propose
the following estimator:
\begin{equation}\label{equ:DKL-fd-approxe}
  \DKLa{\postm}{\priorm} =  
  \frac{1}{2}\Big[ \log\det (\mat{I} + \mat{\Lambda}_1) - \trace(\mat{\Lambda}_1(\mat{I} +\mat{\Lambda}_1)^{-1}) + 
             \cipfd{\dpostmean}{\dpostmean} 
             \Big].
\end{equation}
Since the posterior is Gaussian, the MAP estimator coincides with the
conditional mean, and $\dpostmean$ is obtained by
solving~\eqref{equ:discrete-mean}.
Therefore,
$\cipfd{\dpostmean}{\dpostmean}$ is assumed to be available and need not be
estimated using the eigenvalue decomposition.

The error in the KL-divergence,  defined as follows 
\begin{equation}
	E_\text{KL} \equiv |\DKL{\postm}{\priorm} - \DKLa{\postm}{\priorm}|,
\end{equation} 
can be bounded as 
\begin{equation}\label{e_klerror} 
E_\text{KL}  \leq \frac12 \Big[   \logdet(\mat{I}+\mat{\Lambda}_2) + \trace(\mat{\Lambda}_2)\Big] .
\end{equation}
The proof is a straightforward consequence of the properties of the trace and the determinant and is omitted.

\paragraph{Randomized approach} Computing the exact eigenvalues of $\HMs$ can be computationally challenging. Based on the approach in section~\ref{ss_logdet}, we propose a randomized estimator for the KL divergence, which is cheaper to compute.  
Given the output $\mat{T}$ of Algorithm~\ref{alg:randsvd}, our estimator for the KL divergence is  
%$\HMs \approx \widehat{\mat{U}}\widehat{\mat{\Lambda}}\widehat{\mat{U}}\tran$. Then 
\begin{equation}\label{equ:DKL-fd-approx}
  \DKLa{\postm}{\priorm} =  
  \frac{1}{2}\Big[ \log\det (\mat{I} + \mat{T}) 
	  - \trace(\mat{T}(\mat{I} +\mat{T})^{-1}) + \cipfd{\dpostmean}{\dpostmean} 
             \Big].
\end{equation}
 It is clear that we only 
need to compute the determinant and the trace of an $\ell \times \ell$ matrix
instead of an $n\times n$ matrix. This has significant computational benefits when $\ell \ll n$. 
Theorem~\ref{p_kld} presents the error of the KL decomposition with the
expectation taken over the random starting guess $\mat\Omega$. A similar result
can be derived for concentration, or tail bounds of the distribution, but is
omitted.
%~\alennote{We are focusing on a Gaussian $\mat{\Omega}$, right? we should make
%this clear.}
%We now present a result that quantifies the error in the KL-divergence using the randomized estimator. 
\begin{theorem}\label{p_kld}
Let $\mat{T}$ be computed using Algorithm~\ref{alg:randsvd} with Gaussian starting guess $\mat\Omega \in \mathbb{R}^{n\times (k+p)}$ with $p \geq 2$ and $\DKLa{\postm}{\priorm}$ be computed using~\eqref{equ:DKL-fd-approx}. Then, we have the following estimate for the expected value of the error  
\begin{align} \expect[\mat\Omega]{E_\text{KL}} \leq & \>  \frac12 \left[ (1+ \gamma^{2q-1}C_{\rm ge})\trace(\mat{\Lambda}_2) + \right.\\ \nonumber
	& \qquad \left. \logdet(\mat{I}+\mat{\Lambda}_2) + \logdet(\mat{I}+\gamma^{2q-1}C_{\rm ge}\mat{\Lambda}_2)  \right], 
%\left[|\trace(\HMs) -\trace(\widehat{\mat{\Lambda}})| + 
%$ | \log\det (\HMs + \mat{I}) - \log\det (\widehat{\mat{\Lambda}} + \mat{I})| \right].
\end{align}
and with $\mu = \sqrt{n-k} + \sqrt{k+p}$, the constant $C_{\rm ge}$ is given by 
 $$C_{ge} \equiv \frac{e^2\,(k+p)}{(p+1)^2}\>
 \left(\frac{1}{2\pi(p+1)}\right)^{\frac{2}{p+1}}\>
 \left(\mu+\sqrt{2} \right)^2 \left(\frac{p+1}{p-1}\right). $$
\end{theorem} 
\begin{proof}
See Appendix~\ref{ssec:kld}.
\end{proof}

The result can be further simplified with the observation that for $x \geq 0$, $\log(1+x) \leq x$. Therefore, the upper bound in Theorem~\ref{p_kld} can be simplified to
\[ \expect[\mat\Omega]{E_\text{KL}} \leq \> (1+ \gamma^{2q-1}C_{\rm ge})\trace(\mat{\Lambda}_2).\]
The importance of this result is that the expected error incurred in the KL divergence is no worse than the error in the approximation of $\trace(\HMs)$.  
However, it is worth comparing the result of~\eqref{e_klerror} with
Theorem~\ref{p_kld}. The error in Theorem~\ref{p_kld} is higher; however, the
estimator~\eqref{equ:DKL-fd-approx} is much easier to compute.

\subsection{Efficient computation of OED objective and gradient}\label{ss_obj_grad_all}

To efficiently estimate the OED objective function and gradient, we propose three
different methods: truncated spectral decomposition approach
(section~\ref{ss_obj_grad}); randomized approach (section~\ref{ss_rand_approach}); and
fixed (frozen) low-rank approximation of $\FFp$ (section~\ref{ss_frozen}); then, in
section~\ref{ssec:precomp}, we discuss the computational costs and pros and
cons of each approach. 

%Here we
%work with $\HMs$ instead of $\HMs$. Since the eigenvalues are preserved under
%a similarity transformation, the objective function remains unchanged.

\subsubsection{Spectral decomposition approach} \label{ss_obj_grad}

We show how to compute the objective function $J(\vec{w})$ and its derivatives $\partial_j J(\vec{w})$ using the dominant eigenpairs of $\HMs$. 
%\paragraph{Objective function evaluation}
 Using the spectral decomposition of $\HMs$ in section~\ref{ss_logdet}, we can compute
\[
J(\vec{w}) = \logdet(\mat{I} + \HMs) = \sum_{i = 1}^K \log(1 + \lambda_i).
\]
%Let us denote $\mat{S} = \mat{I} + \HMs$. 

By Sherman-Morrison-Woodbury formula, we have $(\mat{I} + \HMs)^{-1} = \mat{I} - \mat{UDU}\tran$, 
where $$\mat{D} = \diag(\lambda_1/(1+\lambda_1), \ldots, \lambda_K/(1+\lambda_K)).$$
Substituting this expression into the gradient~\eqref{e_partial_J2}, we have, for $j=1,\dots,\Ns$
%(\alennote{What is $\mat{S}$? also would be good to refer to the expression for 
%gradient.}
 \begin{align} \nonumber \partial_j
%J(\vec{w}) &=  \trace( (\mat{I} - \mat{UDU}\tran)\FFp^* \mat{E}^\sigma_j \FFp) \\
J(\vec{w}) &=  \trace( (\mat{I} - \mat{UDU}\tran) \mat{Z}_j) \\
           &= \trace(\mat{Z}_j) -
              \sum_{i=1}^{K} \left[ \frac{\lambda_i}{1+\lambda_i}\left( \vec{u}_i\tran\mat{Z}_j\vec{u}_i\right) \right], 
\end{align} 
where $\mat{Z}_j$ is as in~\eqref{eqn:Zj}. Recall that $\FFp^* = \mat{M}^{-1}\FFp\tran$, therefore, 
\[ \vec{u}_i\tran \mat{Z}_j\vec{u}_i =  \ip{\vec{q}_i}{\mat{E}_j^\sigma\vec{q}_i} \qquad i=1,\dots,K,\]
where $\vec{q}_i = \FFp \mat{M}^{-1/2}\vec{u}_i$. Moreover, we denote 
\begin{equation}\label{eqn:zj}
z_j \equiv \trace(\mat{Z}_j) \qquad j=1,\dots,\Ns;
\end{equation} 
the constants $\{z_j\}_{j=1}^\Ns$ are independent of the weights and can be precomputed; 
see section~\ref{ssec:precomp} for more details. 
To summarize, the gradient can be computed as 
\begin{equation}\label{equ:partialj} 
\partial_j J(\vec{w}) = 
z_j - \sum_{i=1}^K \frac{\lambda_i}{1+\lambda_i} \ip{\vec{q}_i}{\mat{E}_j^\sigma\vec{q}_i},  
\quad j=1,\dots,\Ns.
\end{equation}
%We denote this by 
%\[ z_j =\trace( \mat{M}^{1/2}\FFp^* \mat{E}^\sigma_j \FFp \mat{M}^{-1/2}
%) \qquad j=1,\dots,\Ns.\]
%

The estimators for the D-optimal objective and its gradient, denoted,
respectively, by $\widehat{J}_\text{eig}(\vec{w})$ and $\widehat{\partial
J}_\text{eig}(\vec{w})$ only retain $k \leq K$ eigenpairs\footnote{Here and henceforth, we 
denote the approximate quantities with a $\,\widehat{ }\,$ symbol.}. In what follows, we refer to this 
as the Eig-$k$ approach, where $k$ is the target rank of the approximation.  
The steps for the computation of the D-optimal objective and gradient using 
the Eig-$k$ approach are detailed in Algorithm~\ref{alg:lowrank}.
\newcommand{\Evaluate}{\textbf{Evaluate }}
\newcommand{\Compute}{\textbf{Compute }}
\newcommand{\Solve}{\textbf{Solve }}
\begin{algorithm}[ht]
\caption{Eig-$k$: Truncated spectral decomposition for approximating ${J}(\vec{w})$ and $\nabla {J}(\vec{w})$.} 
\begin{algorithmic}[1]\label{alg:lowrank}
\REQUIRE Constants $\{z_j\}_{j=1}^\Ns$, target rank $k$, and design $\vec{w}$. 
\ENSURE  OED objective $\widehat{J}_\text{eig}(\vec{w})$ and gradient $\widehat{\nabla J}_\text{eig}(\vec{w})$.
\STATE \Solve an eigenvalue problem for eigenpairs $\{(\lambda_i, \vec{u}_i)\}_{i=1}^k$ of $\HMs(\vec{w})$. 
\STATE \Evaluate $\widehat{J}_\text{eig}(\vec{w}) = \sum_{i = 1}^k \log(1 + \lambda_i)$. 
\FOR{$i=1$ to $k$}
   \STATE \Compute $\vec{q}_i = \FFp \mat{M}^{-1/2}\vec{u}_i $% \FFp \vec{v}_i$. 
\ENDFOR
\STATE 
$\widehat{\partial_j J}_\text{eig}(\vec{w}) = z_j - \sum_{i = 1}^k \frac{\lambda_i}{1+\lambda_i}
\ip{\vec{q}_i}{\mat{E}^\sigma_j\vec{q}_i}, \quad j  = 1, \ldots, \Ns$.
\hfill
%\COMMENT{$\mu_i = \lambda_i/(1+\lambda_i)$.}
\end{algorithmic}
\end{algorithm}

%\paragraph{Computational cost}  We discuss the cost of precomputing the
%constants $\{z_j\}_{j=1}^{\Ns}$. For each constant $z_j$, requires applying
%$\FFp \mat{M}^{-1/2}$ to $\Nt$ vectors $\vec{v}_k\otimes \vec{e}_j$; however,
%since $\vec{v}_m\otimes \vec{e}_j$ only has $\Nt$ nonzero columns, the cost of
%precomputing $\{z_j\}_{j=1}^{\Ns}$ is $\Ns\Nt$ PDE solves. 
%Each objective function evaluation requires $K$ eigenpairs of $\HMs$, this
%requires $\mathcal{O}(2k)$ PDE solves. Since the constants
%$\{z_j\}_{j=1}^{\Ns}$ have been precomputed, the gradient only requires an
%additional $K$ PDE solves for computing 
%$\vec{q}_j$.  If $k \ll K$ then this approach will have substantial savings in
%terms of computational costs.

\paragraph{Error analysis} It should be clear that if $k = K$, the objective
function and gradient evaluation is exact. For the case of $k < K$, we record
the following result summarizing the approximation errors in objective function
and gradient.
\begin{theorem}\label{prop:err_exactk}
Let $\widehat{J}_\text{eig}$ and $\widehat{\nabla J}_\text{eig}(\vec{w})$ respectively be the approximation to the OED objective and gradient computed by Algorithm~\ref{alg:lowrank}, with a rank $k$ truncation.
Then, for every $\vec{w} \in \R^\Ns_+$, with $\mat{Z}_j$ defined in~\eqref{eqn:Zj}, and  {for}  $j = 1, \ldots, \Ns$
\begin{align}
|J(\vec{w}) - \widehat{J}_\text{eig}(\vec{w})| = & \>   \log \det(\mat{I} + \mat{\Lambda_2}(\vec{w})) \leq 
\sum_{i = k+1}^K {\lambda_i} \\ 
|\partial_j J(\vec{w}) - \widehat{\partial_j J}_\text{eig}(\vec{w}) | \leq & \> \|\mat{Z}_j\|_2 \sum_{i = k+1}^K \frac{\lambda_i}{1+\lambda_i}.
\end{align}
Here $\lambda_i$ are eigenvalues of $\HMs$ and $\mat\Lambda_2$ is defined in~\eqref{e_evd}. 

\end{theorem}
\begin{proof}See Appendix~\ref{ss_grad}. \end{proof}

\noindent The interpretation of this theorem is as follows: if the truncated
eigenvalues, $\{\lambda_{j}\}_{j=k+1}^K$ are very small the errors are
negligible. In practical computations, we choose an appropriate target rank $k$ 
based on the spectral decay of $\HMs$. 

Theorem~\ref{prop:err_exactk} gives the elementwise error in the gradient. If the error in the $2-$norm of the gradient is desired, it can be obtained from
\begin{equation}\label{eqn:gnrm_exactk}
\|\nabla J(\vec{w}) - \widehat{\nabla J}_\text{eig}(\vec{w}) \|_2 \leq \left(\sum_{j=1}^{\Ns} \norm{\mat{Z}_j}^2\right)^{1/2} \sum_{i = k+1}^K \frac{\lambda_i}{1+\lambda_i}.
\end{equation}
Note also that, due to boundedness of the prior-preconditioned forward operator, 
\{$\norm{\mat{Z}_j}\}_{j=1}^\Ns$ remain bounded with respect to mesh refinements. 

%
%At each iteration of the 
%optimization algorithm, we need to 
%\begin{enumerate}
%\item Solve an eigenvalue problem at a cost of $\mathcal{O}(2 K)$ PDE solves.
%\item Compute $\vec{q}_i = \FFp \vec{v}_i$, $i = 1, \ldots, K$, 
%requiring $K$ PDE solves.  
%\end{enumerate}

\subsubsection{Randomized approach}\label{ss_rand_approach} 
The matrix $\HMs$ is symmetric positive semidefinite and has rapidly decaying eigenvalues
(and is exactly of rank $K = \min\{n,n_sn_t\}$). Under these conditions, the
randomized estimator described in section~\ref{ss_logdet} is likely to be
accurate. Hence, we use it to compute estimate the log-determinant
$\logdet(\mat{I} + \HMs)$.

Let $\mat{T}$ be the output of of
Algorithm~\ref{alg:randsvd} applied to $\HMs$.  The objective function can be
approximated as 
\begin{equation}\label{e_logdetest}
	J(\vec{w}) \approx  \widehat{J}_\text{rand}(\vec{w}) \equiv  \logdet(\mat{I}+\mat{T}).
\end{equation}

Next, we consider approximation of the derivative. 
%Let $\mat{T}_k$ denote the rank-$k$ truncation of $\mat{T}$ and let $\mat{T} =
%\mat{U}_T\widehat{\mat{\Lambda}}\mat{U}^*_T$ be its eigendecomposition. 
%Then,
%we can approximate $\mat{S}$ using 
Following the approach in section~\ref{ss_logdet}, we can approximate 
$\HMs \approx \mat{QTQ}\tran = \widehat{\mat{U}}\widehat{\mat{\Lambda}}\widehat{\mat{U}}\tran$. Then,
using the
Woodbury matrix identity, we have 
\begin{equation}\label{equ:Shat}
(\mat{I} + \HMs)^{-1} \approx (\mat{I} + \widehat{\mat{U}}\widehat{\mat{\Lambda}}\widehat{\mat{U}}\tran)^{-1}
= 
\mat{I}-\widehat{\mat{U}}\widehat{\mat{D}}\widehat{\mat{U}}\tran,
\end{equation}
with
\[ 
%\widehat{\mat{S}} \> =  \> \mat{I}-\widehat{\mat{U}}\widehat{\mat{D}}\widehat{\mat{U}}\tran 
\widehat{\mat{D}} \equiv
\diag( \hat\lambda_1/(1 + \hat\lambda_1), \ldots, \hat\lambda_\ell/(1 + \hat\lambda_\ell) ) \in
\mathbb{R}^{\ell\times \ell}, 
\] 
where $\hat\lambda_i$ are the eigenvalues of $\mat{T}$.  This approximation can
be used to estimate the gradient $\partial_j J(\vec{w})$ as follows.  Write 
\begin{equation}\label{e_grad_est}
\partial_j J(\vec{w}) \approx \widehat{\partial_j J}_\text{rand}(\vec{w}) 
   =  z_j - \sum_{i=1}^{\ell} \left[ \frac{\hat\lambda_i}{1+\hat\lambda_i}
                               \ip{\widehat{\vec{q}}_i}{\mat{E}_j^\sigma \widehat{\vec{q}}_i} \right], 
% \trace(\widehat{\mat{S}}\FFp^*\mat{E}_j\FFp) = 
%\sum_{k=1}^{n_t}  \vec{f}_{jk}^*
%\widehat{\mat{S}} \mat{f}_{jk}.
\end{equation} 
where $\widehat{\vec{q}}_i  =  \FFp \mat{M}^{-1/2}\widehat{\vec{u}}_i$ for $i=1,\dots,\ell$. The 
constants $z_j$ are as before and can be precomputed in the same manner. 
The steps for the computation of the D-optimal objective and gradient using
the randomized approach are detailed in Algorithm~\ref{alg:randobjgrad}.
\begin{algorithm}[ht]
\caption{Randomized method for estimating ${J}(\vec{w})$ and $\nabla {J}(\vec{w})$.} 
\begin{algorithmic}[1]
\REQUIRE Constants $\{z_j\}_{j=1}^\Ns$, target rank $k$, oversampling parameter $p \geq 0$, iteration count $q \geq  1$ and design $\vec{w}$. 
\ENSURE  OED objective $\widehat{J}_\text{rand}(\vec{w})$ and gradient $\widehat{\nabla J}_\text{rand}(\vec{w})$.
\STATE Apply Algorithm~\ref{alg:randsvd} with $\ell = k + p$ and $q$ to obtain $\mat{T} \in \R^{\ell \times \ell}$ and $\mat{Q} \in \R^{n\times \ell}$. 
\STATE \Evaluate $\widehat{J}_\text{rand}(\vec{w}) = \logdet(\mat{I} + \mat{T})$. 
\STATE \Compute spectral decomposition $[\mat{U}_T,\mat{\Lambda}_T] \!\!=\! \text{eig}(\mat{T})$ 
and let $\widehat{\mat{U}} \!=\! \mat{QU}_T$, $\widehat{\mat{\Lambda}} \!=\! \mat{\Lambda}_T$.
\FOR{$i=1$ to $\ell$}
   \STATE \Compute $\widehat{\vec{q}}_i = \FFp \mat{M}^{-1/2}\widehat{\vec{u}}_i $. % \FFp \vec{v}_i$. 
\ENDFOR
\STATE 
$\widehat{\partial_j J}_\text{rand}(\vec{w}) = z_j - \sum_{i = 1}^\ell \frac{\hat\lambda_i}{1+\hat\lambda_i}
\ip{\widehat{\vec{q}}_i}{\mat{E}^\sigma_j\widehat{\vec{q}}_i}, \quad j = 1, \ldots, \Ns$.
%\STATE   \Compute $\partial_j J(\mat{w}) = $.%\sum_{k=1}^\Nt \mat{f}_{jk}^*\widehat{\mat{S}}\mat{f}_{jk}$ where $j=1,\dots,\Ns$. 
\end{algorithmic}
\label{alg:randobjgrad}
\end{algorithm}

\paragraph{Error analysis} 
The error in the objective function approximation, when using the randomized approach, is quantified by 
the following result: 
\begin{theorem}\label{p_logdet} Let $\mat{T}$ be computed using
Algorithm~\ref{alg:randsvd} with Gaussian random starting guess $\mat\Omega \in
\mathbb{R}^{n\times (k+p)}$ with $p \geq 2$ and $\hat{J}(\vec{w})$ be computed
using~\eqref{e_logdetest}. Then, we have the following estimate for the expected value of the error
\begin{equation} 
\expect[\mat\Omega]{| {J}(\vec{w}) -  \widehat{J}_\text{rand}(\vec{w}) |} \leq
\> \logdet(\mat{I}+\mat{\Lambda}_2) + \logdet(\mat{I}+\gamma^{2q-1}C_{\rm
ge}\mat{\Lambda}_2) , 
\end{equation}
 where $C_{\rm ge}$ is defined in
Theorem~\ref{p_kld}.  
\end{theorem} 
\begin{proof} This is a restatement of~\cite[Theorem 1]{SaibabaAlexanderianIpsen17}.  
\end{proof}

Next, we present a result that quantifies the error in the gradient using the
randomized approach.  
\begin{theorem}\label{p_grad} Let $\mat{T}$ be
computed using Algorithm~\ref{alg:randsvd} with starting guess $\mat\Omega \in
\mathbb{R}^{n\times (k+p)}$ with $p \geq 2$ and the approximate derivative
$\widehat{\partial_j J}_\text{rand}(\vec{w})$ defined as~\eqref{e_grad_est}. The expected
error in the approximation to $\partial_j J(\vec{w}) $ is 
\[ \expect{|\partial_j
J(\vec{w})- \widehat{\partial_j J}_\text{rand}(\vec{w})|} \> \leq \> \|
\mat{Z}_j\|_2 \left(1+\gamma^{2q-1}C_{\rm ge}\right) \trace(\mat\Lambda_2), 
\] 
where $C_{\rm ge}$ is defined in
Theorem~\ref{p_kld}, $\mat{Z}_j$ is defined in~\eqref{eqn:Zj} and $j=1,\dots,\Ns$.  
\end{theorem} 
\begin{proof} See Appendix~\ref{ss_grad}. 
\end{proof}
The interpretation of Theorem~\ref{p_grad} is similar to Theorems~\ref{p_kld}
and~\ref{p_logdet}. The estimator is accurate if the eigenvalues contained in
$\mat\Lambda_2$ are small relative to $\mat\Lambda_1$. 
It is worth mentioning that $\widehat{\nabla J}_\text{rand}(\vec{w})$ is not
the exact derivative of the approximate objective function
$\widehat{J}_\text{rand}(\vec{w})$ but its accuracy is comparable to the exact
gradient. This is also illustrated numerically in section~\ref{sec:basic_tests}.
%

%However, compared to Theorem~\ref{p_logdet}, the error bound for the gradient
%is larger by the multiplicative factor $\| \mat{Z}_j\|_2$.

% this factor, however, 
%remains bounded in the infinite-dimensional limit, due to boundedness
%of the parameter-to-observable map.
% and this term remains bounded in the limit of infinite discretization.  

Theorem~\ref{p_grad} gives the component-wise error in the gradient. Under the same assumptions, it can be shown that 
\begin{equation}\label{eqn:gnrm_rand}
 \expect{\|\nabla J(\vec{w})- \widehat{\nabla J}_\text{rand}(\vec{w})\|_2} \leq 
 \left(\sum_{j=1}^{\Ns} \|\mat{Z}_j\|_2 \right)\left(1+\gamma^{2q-1}C_{\rm ge}\right) \trace(\mat\Lambda_2).
\end{equation}
This follows from the vector norm inequality $\|\cdot\|_2 \leq \|\cdot\|_1$ and the
linearity of expectations. Compared to~\eqref{eqn:gnrm_exactk}, this result is
clearly suboptimal. It is not clear, if the bounds can be further tightened.

%\paragraph{Computational complexity}

\subsubsection{Frozen low-rank approximation of $\FFp$}\label{ss_frozen}

The dominant cost in the evaluation of the objective function and the
derivative is the applications of $\FFp$ to vectors. This operator is typically
low-rank for many applications. Since the operator $\FFp$ is independent of the
weights $\vec{w}$, a low-rank approximation to $\FFp$ can be precomputed at the
start of the optimization routine. We refer to this as the ``Frozen low-rank
approach.'' The idea of precomputing a low-rank SVD was used
in~\cite{HaberMagnantLuceroEtAl12,AlexanderianPetraStadlerEtAl14}, for
computing A-optimal designs. Here we show how this can be used efficiently for 
D-optimal experimental design, and provide a result quantifying the approximation error. 

The matrix $\FFp$ is approximated by the rank-k thin SVD $\widehat{\FFp} \equiv \widehat{\mat{U}}\widehat{\mat{\Sigma}}\widehat{\mat{V}}^*$, where $\widehat{\mat{U}} \in \R^{\Ns\Nt \times k}$ has orthonormal columns, $\widehat{\mat\Sigma} \in \R^{k\times k}$ and $\widehat{\mat{V}} \in \R^{n\times k}$ has columns that are orthonormal with respect to the mass-weighted inner product. 
With this low-rank approximation, the approximate objective function is defined to be  
\[ 
\widehat{J}_\text{froz}(\vec{w})  = \logdet(\mat{I} + \widehat{\FFp}^*\mat{W}\widehat{\FFp}). 
\]
Evaluating $\widehat{J}_\text{froz}(\vec{w})$ still requires computation of the
log-determinant. Fortunately, the low-rank approximation can be used to
evaluate this function efficiently.  
\begin{proposition}
The objective function can be computed as 
\[ \widehat{J}_\text{froz}(\vec{w}) = \log\det(\mat{I} + \widehat{\mat{\Sigma}}\tran\widehat{\mat{U}}\tran\mat{W} \widehat{\mat{U}}\widehat{\mat{\Sigma}}). \]
\end{proposition}
\begin{proof}
The formulas $\widehat{\mat{V}}^* = \mat{V}\tran\mat{M}$ and  $\widehat{\FFp}^* = \widehat{\mat{V}}\widehat{\mat{\Sigma}}\tran\widehat{\mat{U}}\tran$ combined with the determinant identity~\cite[Corollary 2.1]{Ou81} give the desired result. 
\end{proof} 
Note that 
computing the objective function using the Frozen low-rank approach only
requires the evaluation of the determinant of an $\ell \times \ell$ matrix, 
instead of an $n\times n$ matrix.
The computation of the derivative, denoted by
$\widehat{\nabla {J}}_\text{froz}(\vec{w})$, is similar to the discussion in the exact
(and randomized) case and is omitted.
We derive the error in the Frozen low-rank approximation when $\widehat{\FFp}$ is the best rank$-k$ approximation to $\FFp$.   
\begin{theorem}\label{thm:froz} 
Let $\widehat{\FFp} = \FFp_k$ be the rank$-k$ truncated singular 
value decomposition to $\FFp$. Then, assuming $w_i \in [0, 1]$, $i = 1, \ldots, \Ns$,   
\[ |{J}(\vec{w}) - \widehat{J}_\text{froz}(\vec{w})| \leq \logdet(\mat{I} + \mat{\Sigma}_2^2),  \]
	with $\mat\Sigma_2$ the diagonal matrix that contains the singular values of $\FFp$ discarded from $\FFp_k$.
\end{theorem}
\begin{proof}
See Appendix~\ref{apdx:froz}.
\end{proof}
The interpretation of this theorem is that the Frozen low-rank approach is
accurate so long as the discarded singular values of $\FFp$ are small. Moreover, this
result is notable in the sense that the error is independent of the weight
vector. 

The precomputation of $\widehat{\FFp}$ is the dominant cost of this algorithm
and requires $\mathcal{O}(2k)$ PDE solves. The main computational 
benefit of this approach is that during the optimization iterations no PDE
solves are necessary, making this algorithm very fast in practice.  The
disadvantages are that the errors in the objective function and gradient computation cannot
be controlled since $\widehat{\FFp}$ is fixed; moreover, storing the low-rank
approximation can be expensive. The other major disadvantage is that this
approach is specific to our formulation of the optimal sensor placement problem. 
The assumption that $\FFp$ is independent of the
design parameters may not hold in general. The Eig-$k$ and the randomized
approaches do not make this assumption and are thus more generally
applicable.   Despite its shortcomings, the Frozen low-rank approach can be viable and
very efficient for some problems. 
Finally, we note that a low-rank approximation of $\FFp$ can 
also be used for fast computation of the KL divergence~\eqref{equ:DKL-fd-alt}.

\subsubsection{Computational costs}\label{ssec:precomp}

%\paragraph{Precomputation} 
Here we discuss precomputation of the constants $\{z_j\}_{j=1}^{\Ns}$
defined in~\eqref{eqn:zj}, and the overall computational cost of our
proposed methods.

 \paragraph{Precomputing $\{z_j\}$} The identity matrix can be expressed as the sum of outer products of vectors
$\{\vec{v}_m\}_{m=1}^{\Nt}$ which form the standard canonical basis for
$\R^{\Nt}$, i.e., $\mat{I}_{\Nt} = \sum_{m=1}^{\Nt} \vec{v}_m\vec{v}_m\tran$.
Then with $\vec{e}_j$ and $\mat{E}_j$ defined in~\eqref{equ:W}
\[ \mat{E}_j  = \sum_{m=1}^{n_t} (\vec{v}_m\otimes \vec{e}_j)
(\vec{v}_m\otimes \vec{e}_j)\tran \in \R^{\Nd \times \Nd}. \] 
Using the cyclic property of the trace and with $\mat{f}_{jm}\equiv \FFp \mat{M}^{-1/2}
(\vec{v}_m\otimes \vec{e}_j)$ 
\begin{equation}\label{eqn:zjcomp} 
z_j = \sigma_j^{-2}\sum_{m=1}^{n_t} (\vec{v}_m\otimes \vec{e}_j)\tran ( \FFp
	\mat{M}^{-1/2} )\tran \FFp \mat{M}^{-1/2} (\vec{v}_m\otimes \vec{e}_j)
	= \sum_{m=1}^{\Nt} \mat{f}_{jm}\tran \mat{f}_{jm}.  
\end{equation}
%
%We discuss the cost of precomputing the constants $\{z_j\}_{j=1}^{\Ns}$. For
%
Computing each $z_j$ requires applying $\FFp \mat{M}^{-1/2}$ to $\Nt$ matrices
$\vec{v}_m\otimes \vec{e}_j$; however, since $\vec{v}_m\otimes \vec{e}_j$ only
has $\Nt$ nonzero columns, precomputing $\{z_j\}_{j=1}^{\Ns}$ costs
$\Ns\Nt$ PDE solves.

%For the derivative,
%observe that  $\widehat{\mat{S}}\mat{f}_{jk}$ can be evaluated efficiently in
%$\mathcal{O}(nk)$ flops, so that $\partial_j J(\vec{w})$ can be evaluated in an
%additional $\mathcal{O}(knn_t)$ flops.

\paragraph{Summary of Computational Costs} 
%We now summarize the computational cost of evaluating the objective function
%and its derivative for the truncated spectral approach. 
Here we discuss the computational cost of objective and gradient evaluation
using our proposed approaches.  For the spectral approach, the objective
function evaluation requires $\mathcal{O}(k)$ matrix-vector products (matvecs) involving the transformed
forward operator $\FFp$, $\mathcal{O}(k)$ matvecs involving its adjoint
$\FFp^*$; derivative evaluation requires $k$ PDE solves and an
additional computational cost of $\mathcal{O}(k^2n)$ flops. The computational
cost of the randomized approach is similar. The computational cost of our
proposed methods is summarized in Table~\ref{tab:summary}.

\begin{table}[!ht]\centering
\begin{tabular}{c|c|c|c}
Method/Component & Forward solves & Adjoint solves & Precomputation \\ \hline 
Naive & $n$ & $n$ &  - \\ 
Eig-$K$ & $K$ & $K$ & -\\ \hline 
Eig-$k$ & $\mathcal{O}(k)$ & $\mathcal{O}(k)$ & $K$  \\ 
Randomized & $\ell(q+2) $ & $\ell(q+1)$ &   $K$ \\
Frozen  & - & - & $\mathcal{O}(k)$
\end{tabular}
\caption{Summary of computational costs measured in terms of PDE solves. Here
$K = \min\{\Ns\Nt,n\}$ and $k \leq K$ is the target rank. Furthermore, $\ell =
k+p$, where $p \geq 0$ is an oversampling parameter and $q$ is the number of
subspace iterations (see Algorithm~\ref{alg:randsvd}).}
\label{tab:summary}
\end{table}

\paragraph{Discussion}A few remarks regarding Table~\ref{tab:summary}:
\begin{enumerate} \item The first two methods, ``Naive'' and Eig-$K$ involve  
no approximations. In the ``Naive'' method, the matrix $\HMs$ is
constructed explicitly, using which the objective function and the gradient are
computed. By contrast, the Eig-$K$ approach only computes the nonzero
eigenpairs of the matrix. The remaining three methods have some approximation
built into them. 

\item The three approximate methods require some form of precomputation. The
Frozen approach requires $\mathcal{O}(k)$ PDE solves, whereas the Eig-$k$ and
Randomized methods require $K$ PDE solves (to precompute $\{z_j\}$); these
costs can be lowered to $\mathcal{O}(k)$ by using a low-rank approximation
$\widehat{\FFp}$ instead of $\FFp$.

\item It was mentioned earlier that the derivative computations require an
additional $\mathcal{O}(k)$ PDE solves. However, if the approximate truncated
spectral representation is computed using either a Krylov subspace method, or
the randomized approach, then the information content in the intermediate steps
can be reutilized to avoid the additional PDE solves. 
 
\item In the Eig-$k$ approach, the $k$ eigenpairs are computed using an
iterative matrix-free method such as a Krylov subspace solver. 
%(We use MATLAB's\verb|eigs| command). 
The asymptotic computational cost of the randomized
approach is comparable with that of Eig-$k$ approach. However, in the
randomized subspace iteration, the matvecs involving $\HMs$ are readily
parallelizable. In contrast, Krylov subspace methods are inherently sequential. 

\item Based on the overall number of PDE solves, the Frozen approach is very
attractive provided the low-rank approximation $\widehat{\FFp}$ remains
sufficiently accurate throughout the optimization iterations. 

\item The computational costs in Table~\ref{tab:summary} assume that the target
rank is the same for all the three approaches, but does not say anything about
the accuracy of each approach. The Frozen approach may overestimate the rank
$k$ needed to accurately approximate $\FFp$. To illustrate this, suppose during
the iteration history (or at the optimal point), only one sensor is active.
This means that the formal rank of $\FFp$ is exactly $K$, but the rank of
$\HMs$ is $\Nt$. The Eig-$k$ and randomized approaches, on the other hand, target the
rank of $\HMs$ rather than $\FFp$ and therefore, can be much more accurate
with a lower computational cost.   

\item In the spectral and randomized methods, the target rank $k$ remains fixed
throughout the optimization iterations.  Development of variations of these
approaches, in which the rank is adaptively adjusted during the optimization
iterations, is subject of our future work. 

\end{enumerate} 

Each of the proposed methods have their own advantages and disadvantages. The
choice of the method used would depend on the context and the specific
application under consideration.  We advocate the randomized approach since it
provides an excellent balance of accuracy, flexibility, efficiency, and ease of
implementation.

%\begin{itemize}
%\item recount the advantages of the randomized trace estimator, ease of implementation, 
%parallelization, etc. 

%\item randomized approach provides the most amount of flexibility and
%computational ease: we can control errors during computations, instead of
%committing to a low-rank $\FFp$ a priori; also in a most extreme case, one can
%still do Frozen $\FFp$, and use randomized (as opposed to spectral) approach,
%and still reap the computational benefits of the randomized approach.

%\item We can mention that our general approach opens the door to approximating
%$z_j$, e.g., with randomized estimators.

%\item Combining Frozen $\FFp$ with our randomized method can be suggested as an alternative
%practical approach for very computationally expensive problems?
%\end{itemize}

\subsection{The optimization problem for finding D-optimal designs}\label{ssec:dopt}
The previous subsections were concerned with efficient methods for computation 
of the D-optimal criterion and its derivatives with respect to the design weights.
Here we formulate the optimization problem to be solved for finding D-optimal designs.
Let $J(\vec{w})$ be as in~\eqref{equ:Jw}. We formulate the optimization problem for finding 
a D-optimal experimental design as follows:
\begin{equation}\label{equ:optim_problem}
   \min_{\vec{w} \in \mathcal{W}} -J(\vec{w}) + \gamma P(\vec{w}),
\end{equation}
where $\mathcal{W} = [0, 1]^\Ns$, $\gamma > 0$, and $P(\vec{w})$ is a
sparsifying penalty function.  A typical approach is to 
to use an $\ell_1$-norm penalty;  in this case, since components of $\vec{w}$
are non-negative, an $\ell_1$-penalty simplifies to $\P(\vec{w}) =
\sum_{j=1}^\Ns w_j$.  

To explicitly enforce binary designs we adopt two strategies. The first
strategy is to use an $\ell_1$-norm penalty function, and then 
threshold the computed optimal weights based on a heuristic; that is, 
we place sensors only in locations whose corresponding weight is above 
a given threshold.  
The second strategy is to solve a sequence of optimization problems with
penalty functions $\{P_\eps\}_{i=1}^{n_\text{cont}}$ that successively
approximate the $\ell_0$-``norm'' as done in~\cite{AlexanderianPetraStadlerEtAl14}. 
In the continuation approach, each successive optimization problem uses
the result of the preceding problem as initial guess---a process commonly referred
to as ``warm starting.''   

In our computations, we solve the OED optimization problem using
\textsc{Matlab}'s interior-point solver provided by the \verb+fmincon+
function;  the objective function and its gradient, computed using our
algorithms, are supplied to the optimization solver, and BFGS approximation to
the Hessian is used.

\section{Model problem}\label{sec:model}
In this section, we describe the model problem used to illustrate our proposed OED methods.
The forward problem is a time-dependent advection-diffusion
equation, and the inverse problem is the inference of the initial state
from sensor measurements of the state variable at discrete points in
time; see
also,~\cite{AkcelikBirosDraganescuEtAl05,FlathWilcoxAkcelikEtAl11,PetraStadler11,AlexanderianPetraStadlerEtAl14}.
The setup of the model problem below is mainly based on~\cite{AlexanderianPetraStadlerEtAl14}.

\paragraph{The governing PDE and the parameter-to-observable map}

Let $\D \subset \R^2$, be a two-dimensional domain 
depicted in Figure~\ref{fig:domain_and_sensors}(left).  The domain boundaries $\partial \D$
include the outer edges as well as the internal boundaries of the rectangles
that model buildings/obstacles. Given an initial state $\ipar$, we solve  a time-dependent 
advection-diffusion equation
%following PDE 
for the state variable $u(\vec{x}, t)$: 
\begin{equation}\label{eq:ad}
  \begin{aligned}
    u_t - \kappa\Delta u + \mathbf{v}\cdot\nabla u &= 0 & \quad&\text{in
    }\D\times (0,T), \\
    u(\cdot, 0) &= \ipar  &&\text{in } \D , \\
    \kappa\nabla u\cdot \vec{n} &= 0 &&\text{on } \partial\D \times (0,T),
  \end{aligned}
\end{equation}
where, $\kappa > 0$ denotes the diffusion coefficient, and $T > 0$ is the final
time. We use $\kappa = 0.001$ and $T = 5$ in the numerical experiments below. 
The velocity field $\vec{v}$, shown in Figure~\ref{fig:domain_and_sensors}(left), is computed by
solving a steady Navier-Stokes equation as 
in~\cite{AlexanderianPetraStadlerEtAl14,PetraStadler11}. 
%\begin{equation}\label{eq:fp:navst}
%  \begin{aligned}
%    - \frac{1}{\operatorname{Re}} \Delta \vec{v} + \nabla q + \vec{v} \cdot \nabla
%    \vec{v} &= 0 &\quad&\text{ in
%    }\D,\\
%    \nabla \cdot \vec{v} &= 0 &&\text{ in }\D,\\
%    \vec{v} &= \vec{g} &&\text{ on } \partial\D.
%\end{aligned}
%\end{equation}
%Here, $q$ is pressure, $\text{Re}$ is the Reynolds number,
%set to $\text{Re} = 100$ in the present examples. The Dirichlet boundary data
%$\vec{g} \in \R^2$ is given by 
%$\vec{g} = (0,1)\tran$ on the left boundary of the domain, 
%$\vec{g}= (0,-1)\tran$ on the right boundary,  and $\vec{g} = \vec{0}$ everywhere else.
%%
\begin{figure}\centering
\includegraphics[width=0.35\textwidth]{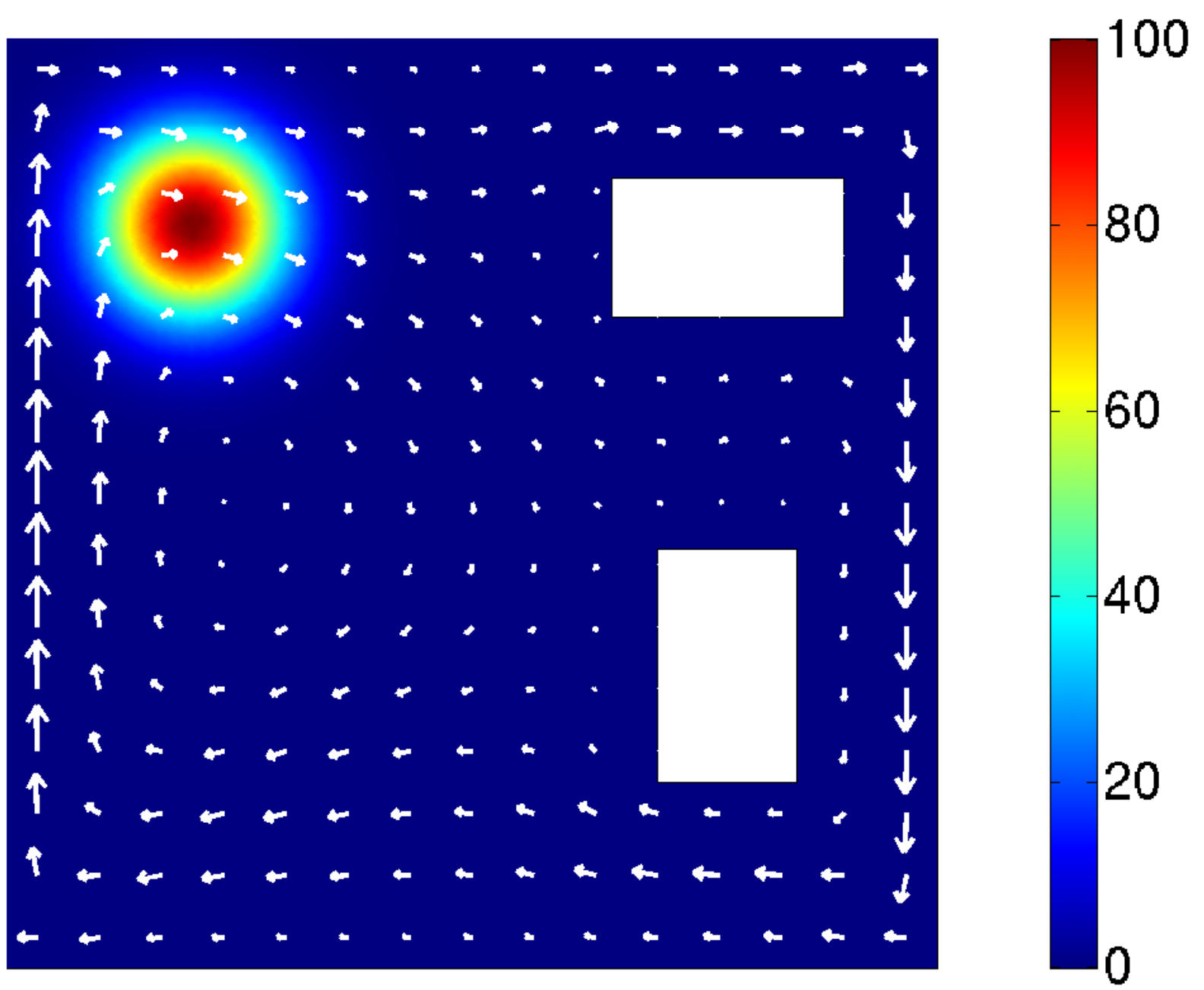}
\includegraphics[width=0.28\textwidth]{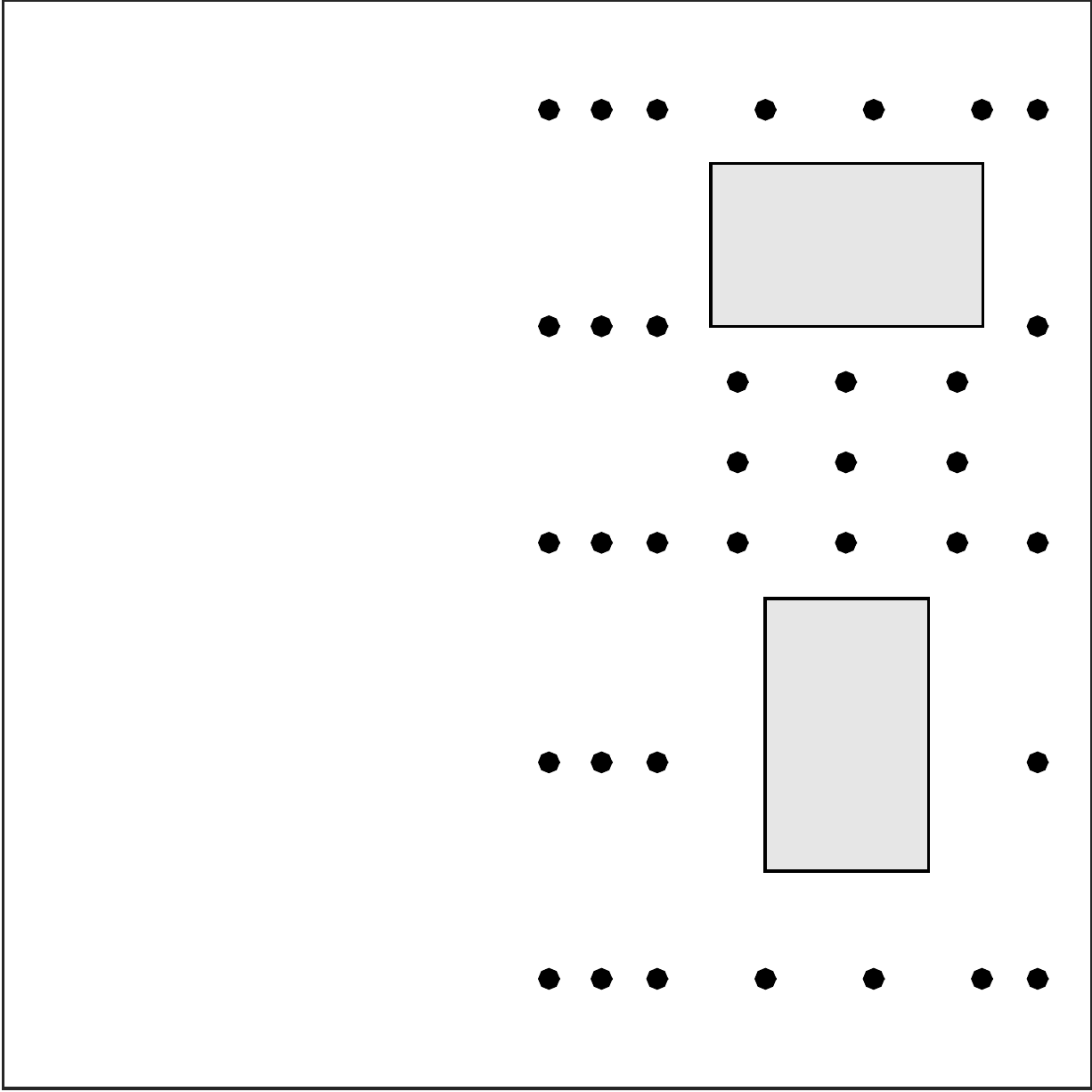}
\includegraphics[width=0.28\textwidth]{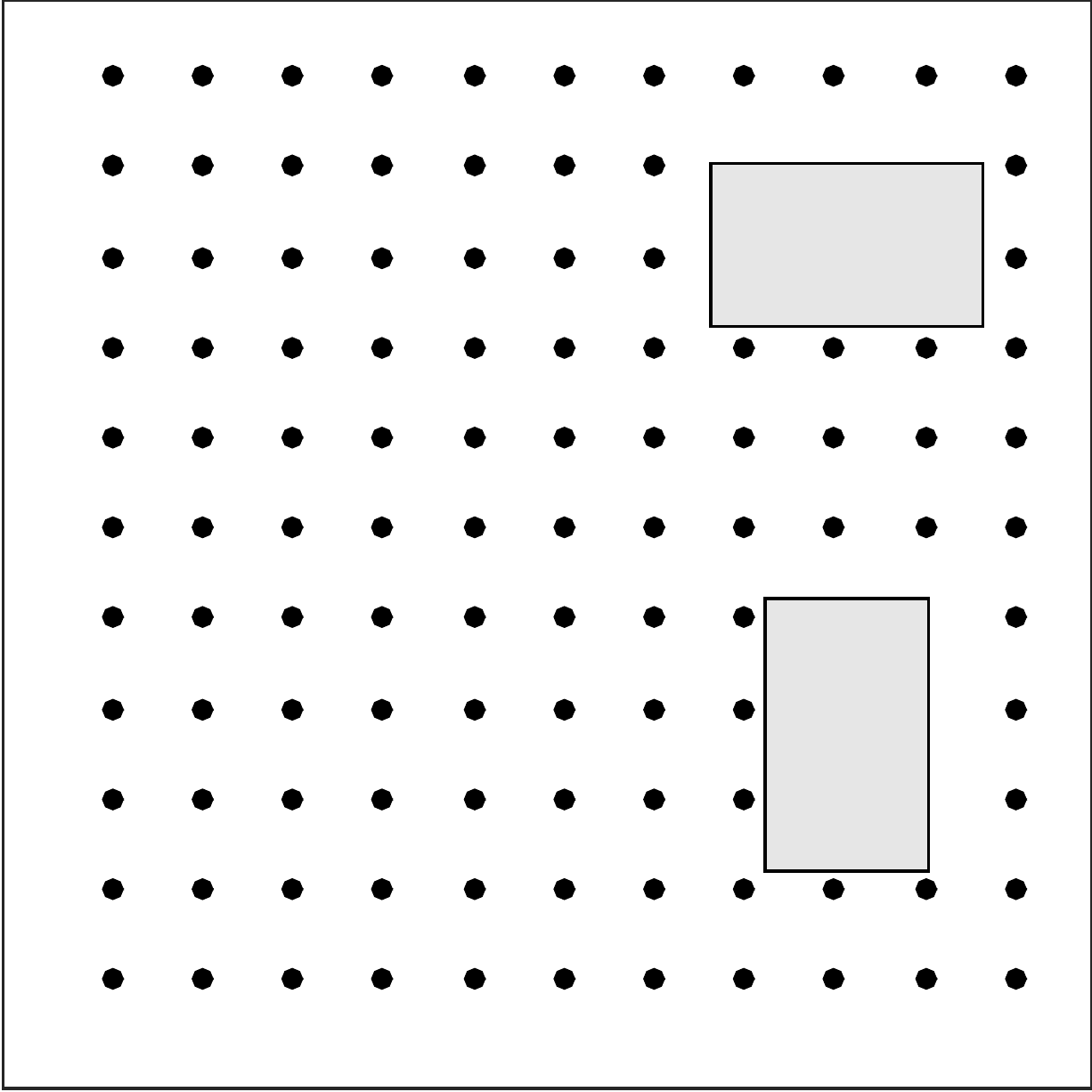}
\caption{The domain $\D$ for the model problem 
is $[0,1]^2$ with the two rectangular regions removed (left); we show the
``true'' initial state and the arrows 
indicate the velocity field.
Two possible grid of candidate sensor locations (middle, right).}
\label{fig:domain_and_sensors}
\end{figure}

The observation operator
denoted by $\mathcal{B}$, returns the values of the state variable $u$ at a set of
sensor locations $\{\vec{x}_1, \ldots, \vec{x}_\Ns\} \subset \D$ at observation
times $\{t_1, \ldots, t_\Nt\} \subset [0, T]$.  Hence, an evaluation of the
parameter-to-observable map, denoted by $F$ in the infinite dimensions, 
involves, first solving the time-dependent
advection-diffusion equation \eqref{eq:ad} to obtain $u=u(\ipar)$, and then
applying the observation operator that returns value of $u$ at the measurement
locations and times. That is, ${F}\ipar = \obs$, where $\obs \in \R^\Nd$ is
given by, $\obs = [\vec{y}_1\tran, \vec{y}_2\tran, \ldots, \vec{y}_\Nt\tran]\tran$, where
$y_j^i$ is the value of $u$ at $\vec{x}_i$ and at observation time $t_j$. 

In what follows, we also need the action 
of the adjoint ${F}^*$ of the parameter-to-observable map ${F}$ to vectors.
For a given observation vector $\dd \in \R^\Nd$, ${F}^*\dd$ is computed by solving the
\emph{adjoint equation} (see~\cite{AkcelikBirosDraganescuEtAl05,
  FlathWilcoxAkcelikEtAl11, PetraStadler11}) for the adjoint variable
$p = p(\vec{x}, t)$,
\begin{equation}\label{eq:ad:adj}
  \begin{aligned}
    -p_t - \nabla \cdot (p \vec{v}) - \kappa\Delta p  &= -\obsop^* \dd %\ncov^{-1}(\obsop u - \obs)
&\quad&\text{ in
    }\D\times (0,T),\\
    p(\cdot, T) &= 0 &&\text{ in } \D,  \\
    (\vec{v}p+\kappa\nabla p)\cdot \vec{n} &=  0 &&\text{ on }
    \partial\D\times (0,T),  
  \end{aligned}
\end{equation}
and setting ${F}^*\dd = -p(\cdot, 0)$.  
Note that this equation is a final value problem that is solved backwards in time.

\paragraph{Bayesian inverse problem formulation}
The Bayesian inverse problem seeks to infer the initial state $\ipar$ from
point measurements of $u$. Following the setup outlined in
section~\ref{sec:background}, we utilize a Gaussian prior measure $\priorm =
\GM{\iparpr}{\Cprior}$, with $\Cprior$ as described before, and use an additive
Gaussian noise model.  Specifically, we use $\alpha = 2\times 10^{-3}$ and
$\beta = 10^{-1}$ in~\eqref{equ:prior_sqrt}.  The ``true'' initial state shown
in Figure~\ref{fig:domain_and_sensors}(left) is used to synthesize data, and the noise level is set as
follows. We solve the forward model using the true initial state and record
measurements of the state at the candidate sensor locations and at the
observation times (see numerical results section for specifics), and define the
noise level as two percent of the maximum of the recorded measurements.

%The solution of the  Bayesian inverse
%problem is given by the posterior law of $\ipar$, which, 
%owing to linearity of the parameter-to-observable map, is a Gaussian measure 
%$\postm = \GM{\iparpost}{\C_\text{post}}$, and 
%the posterior mean and covariance
%$\iparpost$ and $\Cpost$ are as in~\eqref{equ:mean-cov}.
%%
%Note that the posterior mean $\iparpost$ is the minimizer of 
%\[
%\begin{aligned}
%& \mathcal{J}(\ipar) :=
%  \frac{1}{2} \left\| \mathcal{B}u(\ipar) -\obs  \right\|^2_{\Gamma^{-1}_{\mathrm{noise}}}
%  + \frac 12 \left\| \mathcal{A}(\ipar - \iparpr \right)\|^2_{L^2(\D)},
%\end{aligned}
%\]

\paragraph{Discretization}
We discretize the forward and adjoint problems via linear triangular continuous
Galerkin finite elements in the two-dimensional spatial domain,  and use the
implicit Euler method for time integration.  As
in~\cite{PetraStadler11,AlexanderianPetraStadlerEtAl14}, we follow a
discretize-then-optimize approach, where the discrete adjoint equation is
obtained as the adjoint of the discretized forward equation.

%
% Numerical results
%
\section{Numerical results}\label{sec:numerics}
In this section, we test various aspects of our proposed numerical methods for
D-optimal sensor placement. Section~\ref{sec:basic_tests} is devoted to testing
accuracy of our randomized estimators, and a simple illustration of D-optimal
sensor placements using the randomized approach. In section~\ref{sec:hires} we
consider a more realistic example in which we allow a denser grid of candidate
sensor locations, and the unknown parameters are discretized on a grid of
higher resolution. The D-optimal sensor placement methods are illustrated on
this application. 
%In our computations, the OED optimization problem was solved in \textsc{Matlab}; we use
%the interior-point solver provided by \verb+fmincon+, where we supply the
%objective function and its derivative, computed using our algorithms, and use
%BFGS approximation to the Hessian.

\subsection{Test of accuracy and basic illustrations}\label{sec:basic_tests} In
this subsection, we use the following parameters that specify the test problem: 
we pick the grid of $n_s = 35$ candidate sensor locations, depicted in
Figure~\ref{fig:domain_and_sensors}~(middle), with $n_t = 3$ observation times
given by $t = 1$, $t = 2$, $t = 3.5$, for a total of $105$ observations.  The
discretized parameter dimension for this example is $\Nm = 1018$. 

\paragraph{Accuracy of KL-divergence estimator}
Here we demonstrate the accuracy of the estimators for the
computation of the Kullback-Liebler divergence.
We take $\vec{w}$ to be a $35\times 1$ vector of all ones; i.e.,  all 
the sensors are active.
\begin{figure}[!ht]
\centering
\includegraphics[width=.45\textwidth]{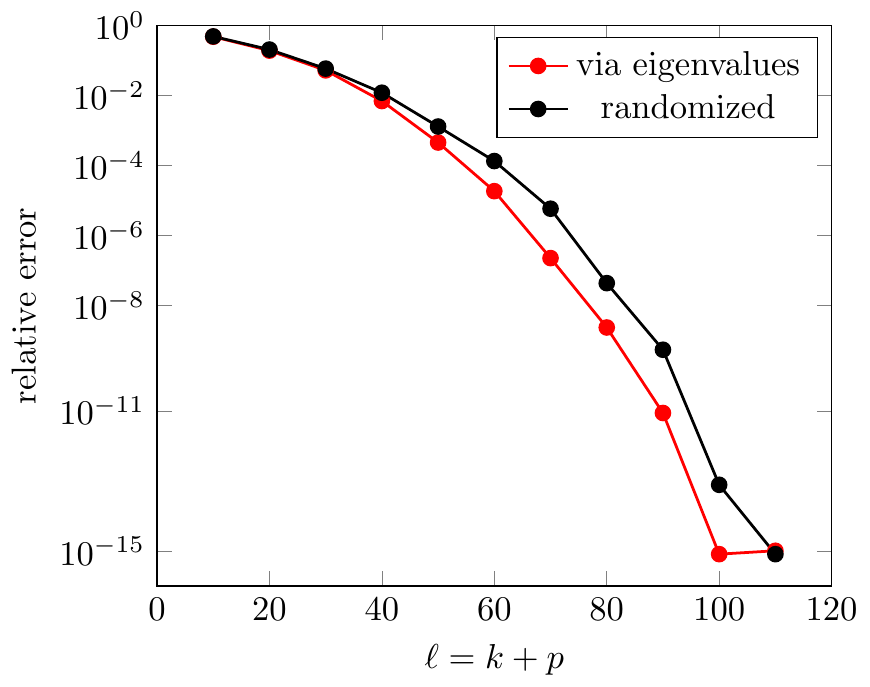}
\caption{Accuracy of KL divergence estimators 
~\eqref{equ:DKL-fd-approxe} (exact eigenvalues) 
and ~\eqref{equ:DKL-fd-approx} (randomized).} 
\label{fig:klaccuracy}
\end{figure} 

%\begin{figure}[!ht] \centering
%\includegraphics[width=.4\textwidth]{./figs/accuracy_DKL.pdf}
%\includegraphics[width=.4\textwidth]{./figs/accuracy_DKL_vs_mesh.pdf}
%\caption{Right: The error in the KL divergence using ``exact''
%eigenvalues~\eqref{equ:DKL-fd-approx} and randomized
%estimator~\eqref{equ:DKL-fd-approx}; right: the convergence of random KL
%divergence estimator as parameter dimension increases.}
%\label{fig:DKL_accuracy} \end{figure} 
We estimate the KL divergence using two different methods proposed in
section~\ref{ss_klde}. In the first method  the KL divergence is computed using
``exact'' eigenvalues~\eqref{equ:DKL-fd-approxe}; the approximation is because
only $k \leq K$ eigenvalues are used for estimating the KL divergence. The
``exact'' eigenvalues were computed using \verb|eigs| function in MATLAB. In
the second method we compute the KL divergence using the randomized
estimator~\eqref{equ:DKL-fd-approx}. The oversampling parameter was fixed to be
$p=5$, we increase the total number of random samples $\ell = k + p$ between $20$
and $105$; since the oversampling parameter is fixed, this amounts to
increasing the target rank $k$. The error in the KL divergence is plotted in
Figure~\ref{fig:klaccuracy}. As we see, the error decays rapidly as the number
of computed eigenvalue increases, and the randomized estimators is nearly as
accurate as using the ``exact'' eigenvalues, but is, in general, considerably
more efficient to compute. 

\paragraph{Accuracy of D-optimal criterion and its derivative} The setup
is the same as the previous experiment; now, we consider the accuracy of the
D-optimal objective function and the gradient. We use Eig-$k$ and randomized 
estimators and the results are plotted in
Figure~\ref{fig:dopt_accuracy}~(left, middle).  Similar conclusions are drawn
here as well: the error decreases with increasing target rank $k$ and the
accuracy of the randomized estimators is comparable with that of Eig-$k$.  
Note that since we used $\mat{W} = \mat{I}$, the frozen and Eig-$k$ approaches
are essentially equivalent, for the present test. 
In our next experiment, we plot the accuracy of the objective
function as a function of mesh refinement in
Figure~\ref{fig:dopt_accuracy}~(right).  The accuracy of the estimators does
not degrade with increasing mesh refinement. This desirable property is also a
numerical illustration of the fact that our formulation remains valid in the
infinite-dimensional limit.

%This also
%demonstrates that our formulations of the expressions for D-optimal criterion
%is stable with respect to increases in parameter dimension, which is to be
%expected, as our formulation remains valid in the infinite-dimensional limit. 

\begin{figure}[!ht]
\centering
\includegraphics[width=1.1\textwidth]{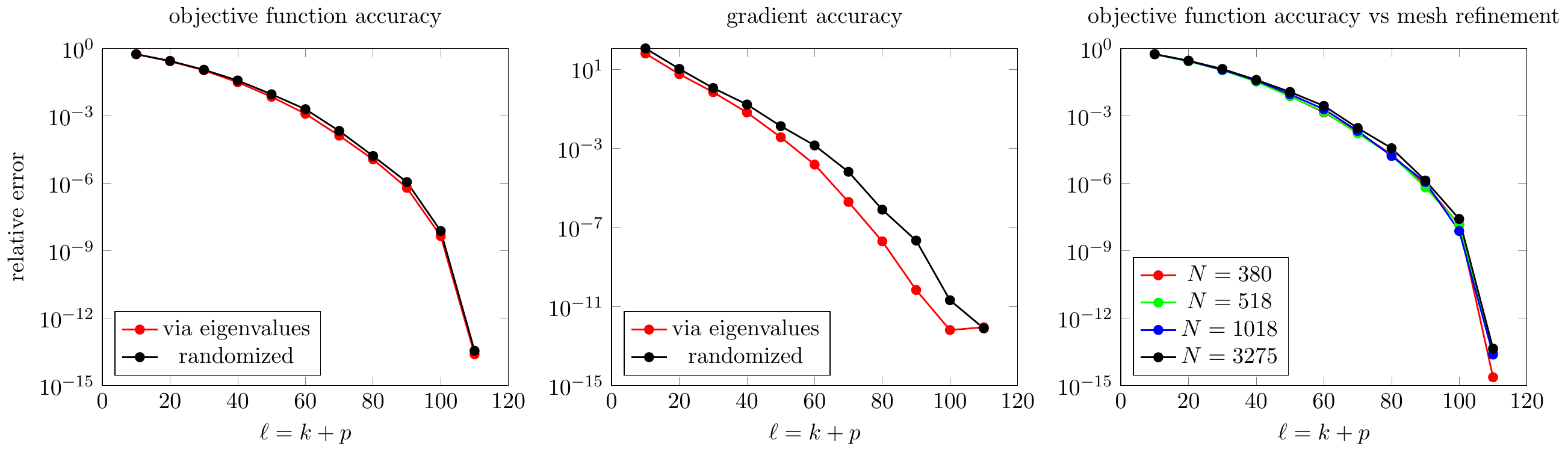}
	\caption{The relative error in the D-optimal objective function (left) and relative error in the $2$-norm of the gradient (center). Performance of randomized estimator with mesh refinement (right). }
\label{fig:dopt_accuracy}
\end{figure}

\paragraph{Computing D-optimal designs} We have demonstrated that our
estimators are accurate; we now show how the corresponding designs look like.
As a first illustration, we use the randomized approach in
section~\ref{ss_rand_approach}, where OED objective 
and gradient are computed according Algorithm~\ref{alg:randobjgrad}, to compute a D-optimal
sensor placement.
The $\ell_1$-norm is chosen as the penalty function to ensure sparsification
(see  section~\ref{ssec:dopt}). The resulting design has weights between $0$
and $1$. To enforce binary weights, the following thresholding criterion was
used: if the sensor weight $w_i$ satisfies $w_i /\sum_j w_j \geq 3\times
10^{-2}$ then it is set to be $1$ else $0$. The resulting designs are plotted
in  Figure~\ref{fig:OED_and_accuracy}~(left). 
We also plot the relative errors in objective and gradient
computations during the course of the optimization iteration history in
Figure~\ref{fig:OED_and_accuracy}~(right). The latter shows that the randomized
estimators remain accurate over the course of the iterations.
%For this
%initial illustration, we use an $\ell_1$-norm penalty approach (cf.
%section~\ref{ssec:dopt}); see Figure~\ref{fig:OED_and_accuracy}~(left).  We
%also provide a plot of the relative errors in objective and gradient
%computations over the optimization iterations in
%Figure~\ref{fig:OED_and_accuracy}~(right). The latter shows that the randomized
%estimators remain accurate over the course of the iterations.

\begin{figure}\centering
\includegraphics[width=.4\textwidth]{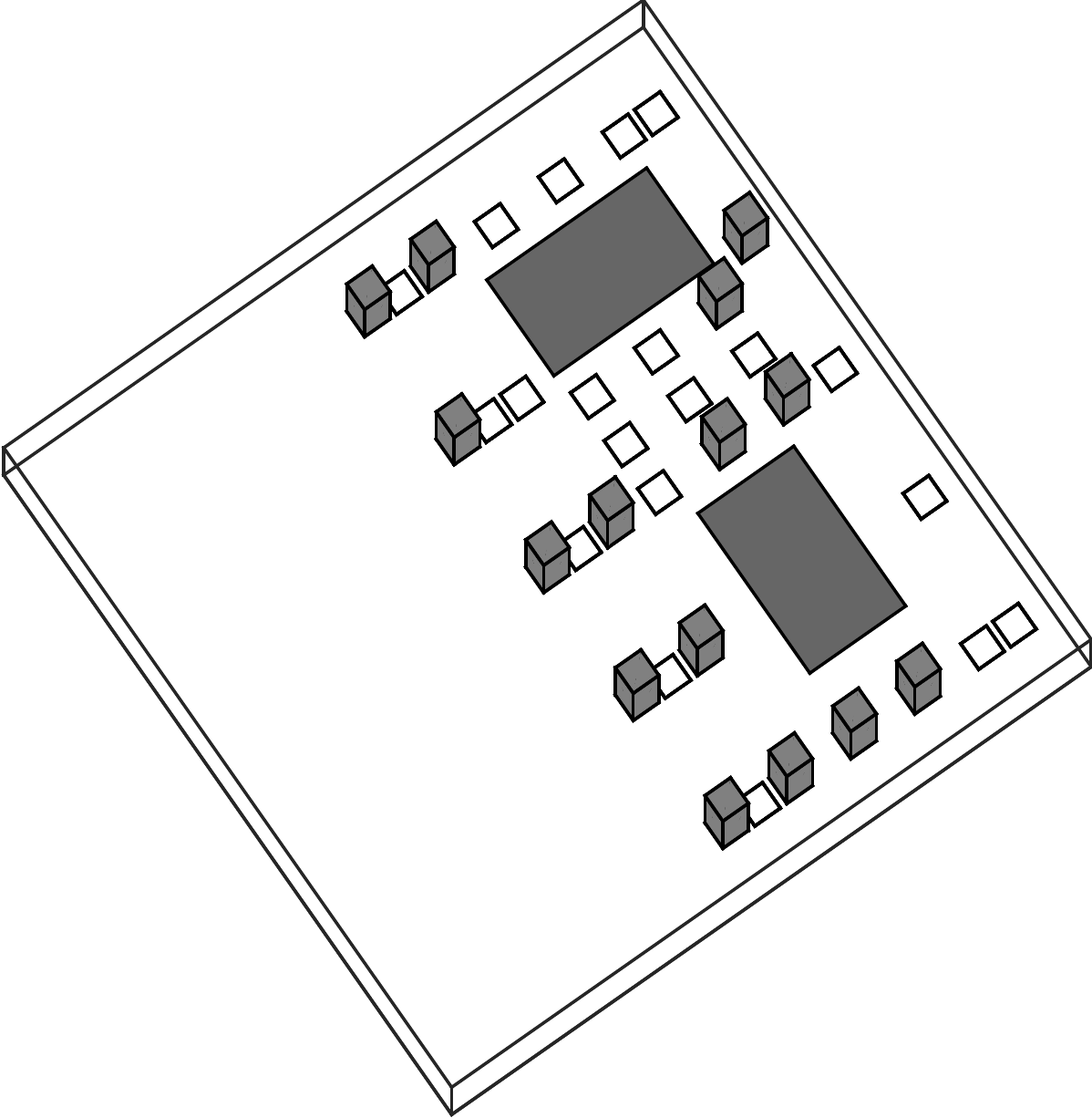}
\includegraphics[width=0.5\textwidth]{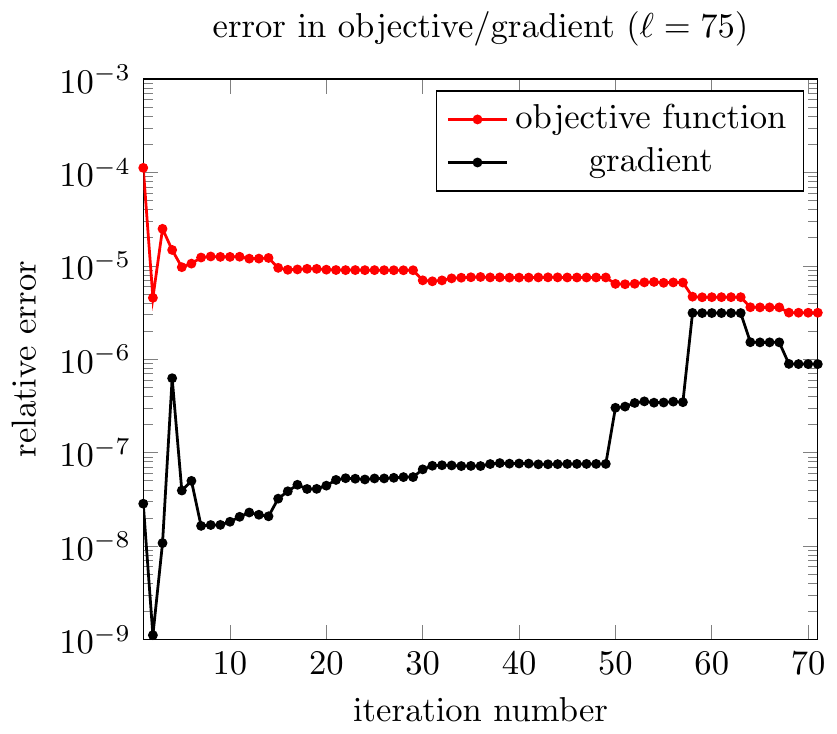}
\caption{Left: D-optimal design with $15$ active sensors; the sensor placement was
obtained using an $\ell_1$-norm penalty approach, with penalty parameter $\gamma = 9$. 
Sensors are placed at locations whose corresponding sensor weight exceeded 
$w_i /\sum_j w_j \geq 3\times 10^{-2}$.
Right: relative error in OED objective and gradient over optimization iterations.}
\label{fig:OED_and_accuracy}
\end{figure}

To compare the different methods proposed in the present work, we compute 
optimal design weights using spectral, randomized, and frozen low-rank approaches;
see Figure~\ref{fig:lowres_comp}. To remain consistent across methods, for each 
method we use rank $k = 70$ approximations with an oversampling parameter of $p = 5$.
We use a log-scale on the vertical axis, and plot the 
computed weights in descending order; reported also are the respective optimal objective 
values. As expected, the randomized and spectral approaches
lead to nearly identical optimal weights; on the other hand, the solution obtained via the
frozen approach, agree with that of the randomized and spectral approach only for the larger
weights.  

\begin{figure}[!ht]
\centering
\includegraphics[width=.75\textwidth]{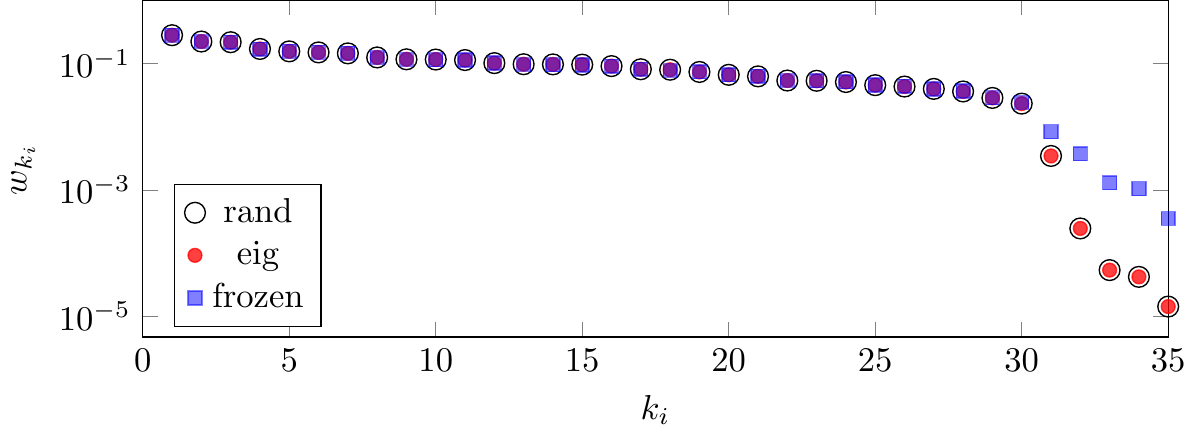}
\caption{
Comparing the optimal design weights obtained using the randomized, spectral, and frozen
low-rank approaches; the objective value at the 
respective solutions were $-72.5847$,  $-72.5848$, $-72.5786$. The problems 
were solved with an $\ell_1$-norm penalty and a penalty parameter of $\gamma = 9$.}
\label{fig:lowres_comp}
\end{figure}

\subsection{A higher resolution example}\label{sec:hires}
Here we illustrate computing a D-optimal sensor placement on a problem 
with parameter dimension $n = 2605$, and a grid of $\Ns = 109$ candidate sensor locations; see
Figure~\ref{fig:domain_and_sensors}(right). 
% Namely, we consider
%a refined parameter dimension of $2605$, and use a denser grid of candidate
%sensor location with $\Ns = 109$ sensors; see
%Figure~\ref{fig:domain_and_sensors}(right).
%
%A continuation strategy of~\cite{AlexanderianPetraStadlerEtAl14} is used to
%enforce sparsification of the weights.  In this approach, first the
%optimization problem~\eqref{equ:optim_problem} is solved with an $\ell_1$-norm
%penalty, and then a sequence of optimization problem 
%
To enforce binary weights, we used the continuation approach described in section~\ref{ssec:dopt}. We picked the
penalty parameter $\gamma = 1.8$ and  the penalty functions
$P_\eps$ is as in~\cite[Section 4.5]{AlexanderianPetraStadlerEtAl14}. We use
six continuation steps with $\{P_{\eps_i}\}_{i=1}^{n_\text{cont}}$, $\eps_i =
1/2^i$.  The resulting D-optimal sensor placement is shown in
Figure~\ref{fig:design_hires}~(left) and the convergence of the
successive design vectors to a binary weight vector is shown in Figure~\ref{fig:design_hires}~(right).  
\begin{figure}[!ht]\centering
\includegraphics[width=.33\textwidth]{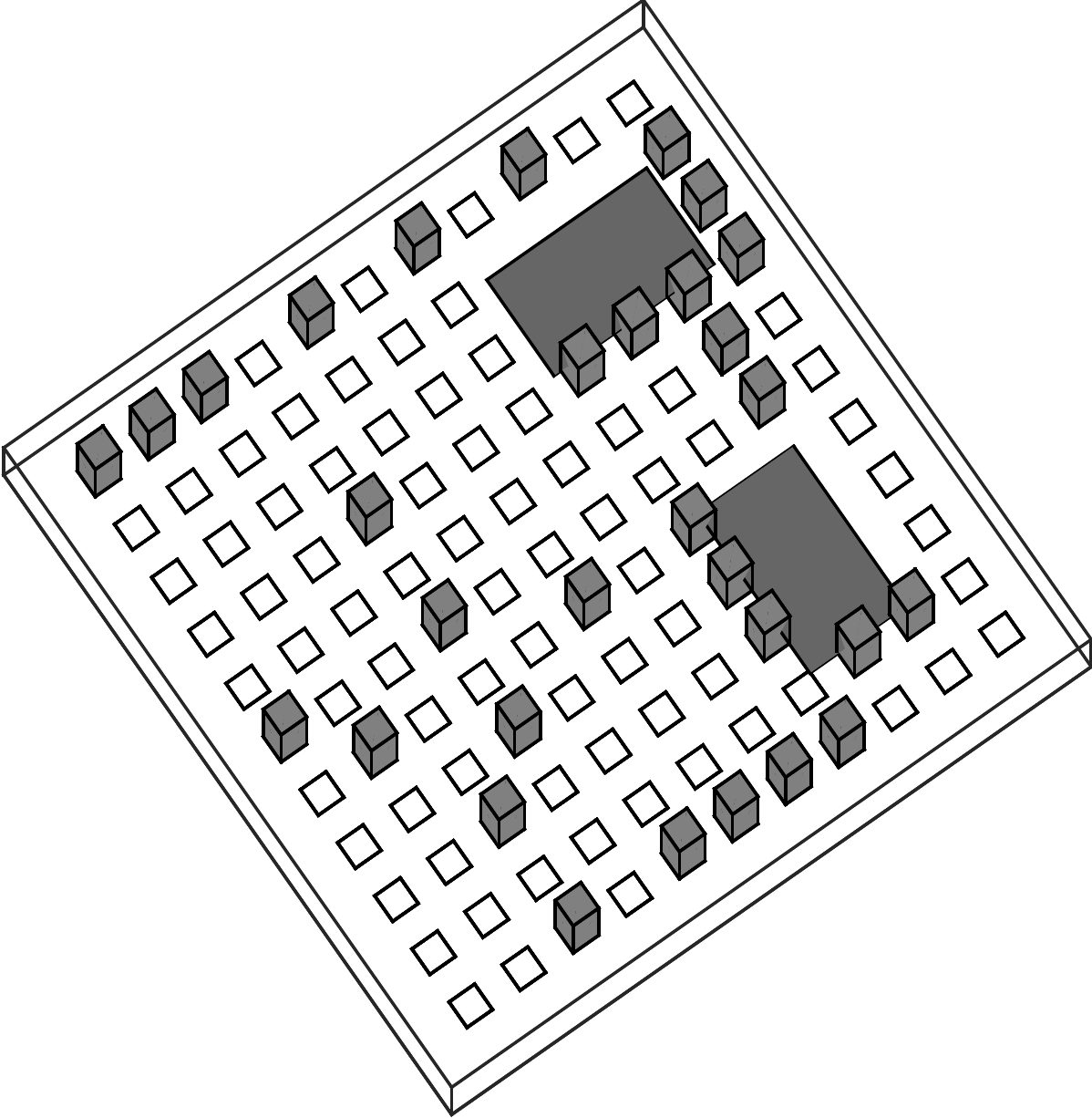}
\includegraphics[width=.33\textwidth]{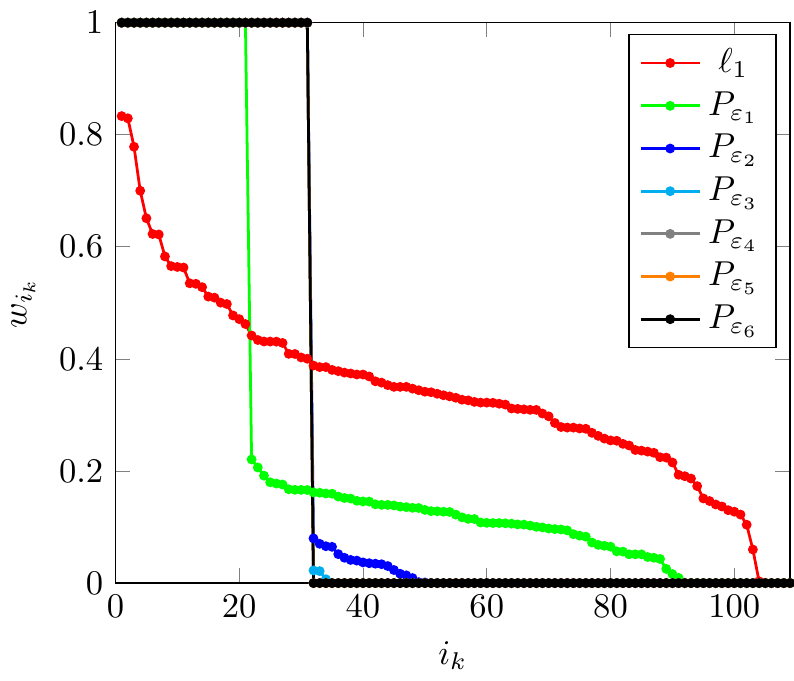}
\caption{Left: a D-optimal design computed on the resolution model problem;
right: the design weights, sorted in descending order, during the
continuation.} \label{fig:design_hires} 
\end{figure}

\paragraph{Testing the effectiveness of the computed D-optimal design} We illustrate
the effectiveness of the computed D-optimal sensor placement by comparing it
with an ensemble of randomly generated sensor placements. Since the
continuation approach yielded $31$ active sensors, the random sensor placements
had the same number of sensors.  For each design (random ones and the optimal
one), we compute the objective value ${J}(\vec{w})$, as
well as the KL divergence from the posterior to prior. To compute the KL
divergence, the inverse problem was solved for each of the sensor placements
and the randomized estimator~\eqref{equ:DKL-fd-approx} was used. 

In the first place, it is expected that the computed optimal design outperform
random designs in terms of smaller OED objective value.  Moreover,
Theorem~\ref{thm:doptimal_criterion} indicates that maximizing
${J}(\vec{w})$ results in
maximizing the expected information gain. Both of these issues are illustrated
in Figure~\ref{fig:cloud}. The red dots in that figure correspond to $500$
randomly generated sensor configurations with $31$ sensors and the black dot to
the computed optimal design. That the black dot falls below all the red dots is
to be expected, because the quantity in the vertical axis, 
$-J(\vec{w}) = -\logdet(\mat{I} + \HMpd(\vec{w}))$, is what we sought to
minimize. On the other hand, the computed D-optimal sensor
placement maximizes the information gain in the average sense specified in
Theorem~\ref{thm:doptimal_criterion}; this explains the location of the
black dot in the horizontal axis. 
%We mention that performing such a test is a reasonable validation test when
%computing D-optimal sensor placements in practical settings. 
\begin{figure}[!ht] \centering
\includegraphics[width=.5\textwidth]{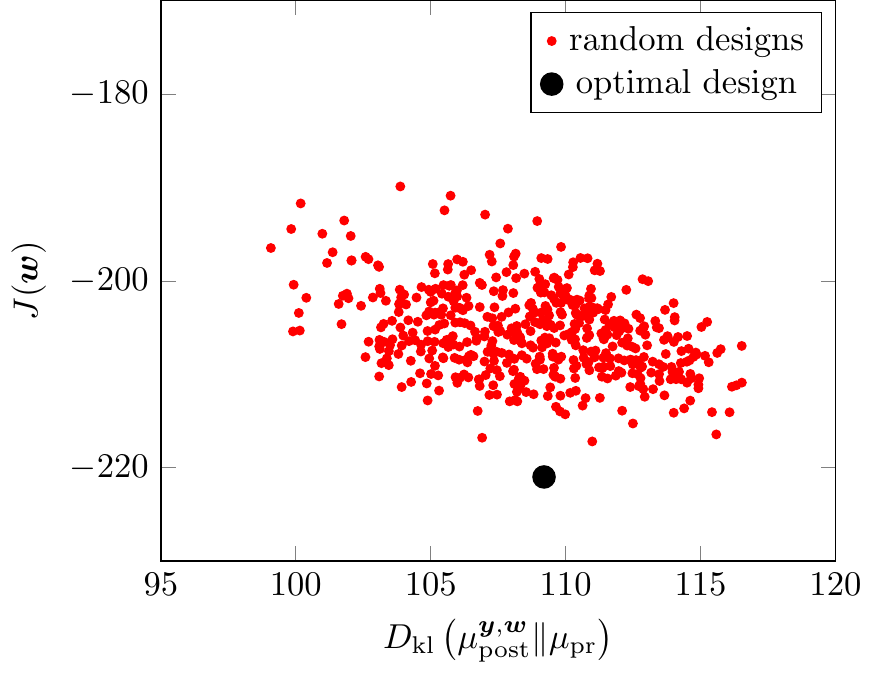} \caption{
Comparing $-J(\vec{w}) =  -\logdet(\mat{I} + \HMpd(\vec{w}))$ and the information gain,
computed at the optimal design (black dot) with $500$ randomly generated
designs (red dots).} \label{fig:cloud} \end{figure}

\paragraph{Solving the inverse problem using the computed optimal design}
Finally, we report the results of solving the Bayesian inverse problem via the computed
D-optimal sensor placement.  Figure~\ref{fig:reconstruct} shows the ``true''
parameter (initial state) and the computed MAP estimator.
Figure~\ref{fig:UQ}~(top) compares the prior and posterior standard deviation
fields, showing reduction in uncertainty and the information gain by using an
optimal sensor placement.  Figure~\ref{fig:UQ}~(bottom) shows three samples
drawn from the posterior distribution.

\begin{figure}[!ht]\centering
\includegraphics[width=.4\textwidth]{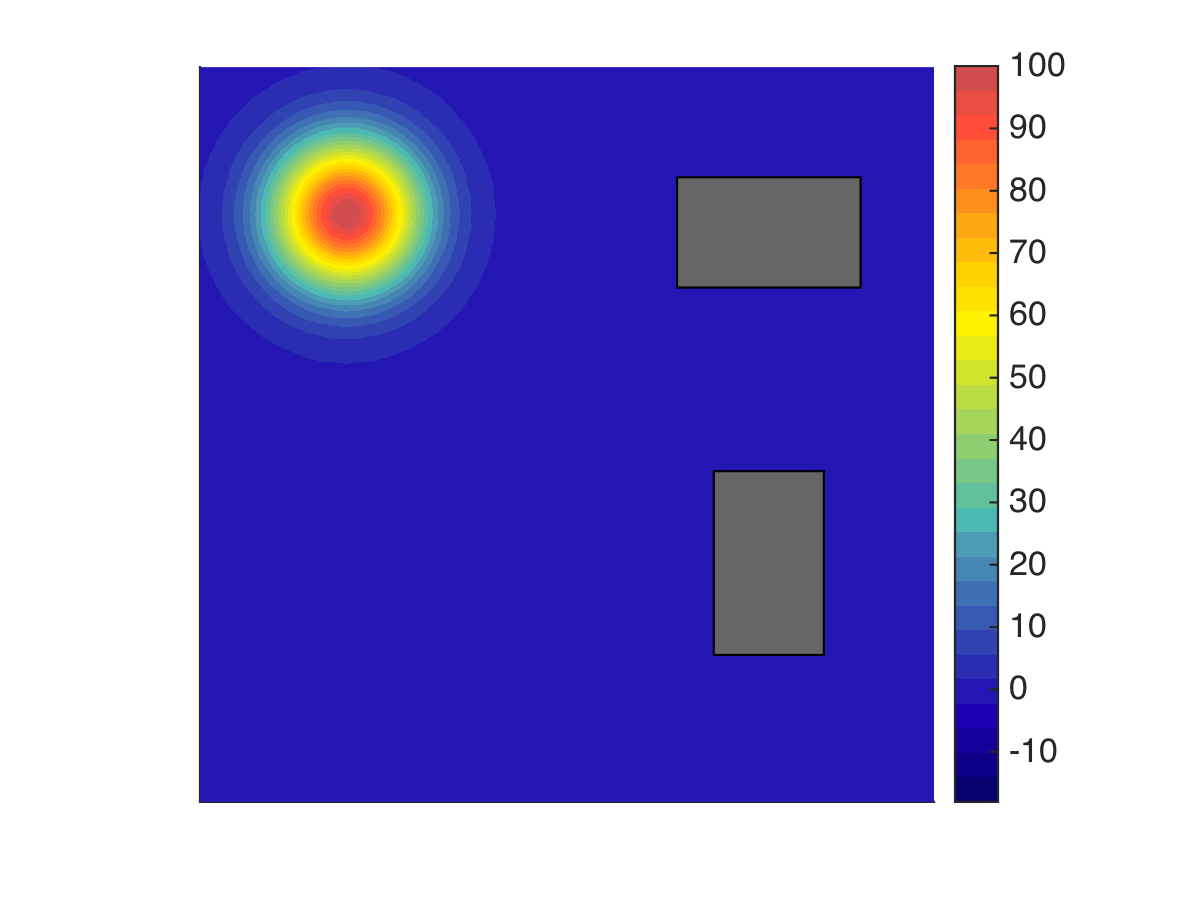}
\includegraphics[width=.4\textwidth]{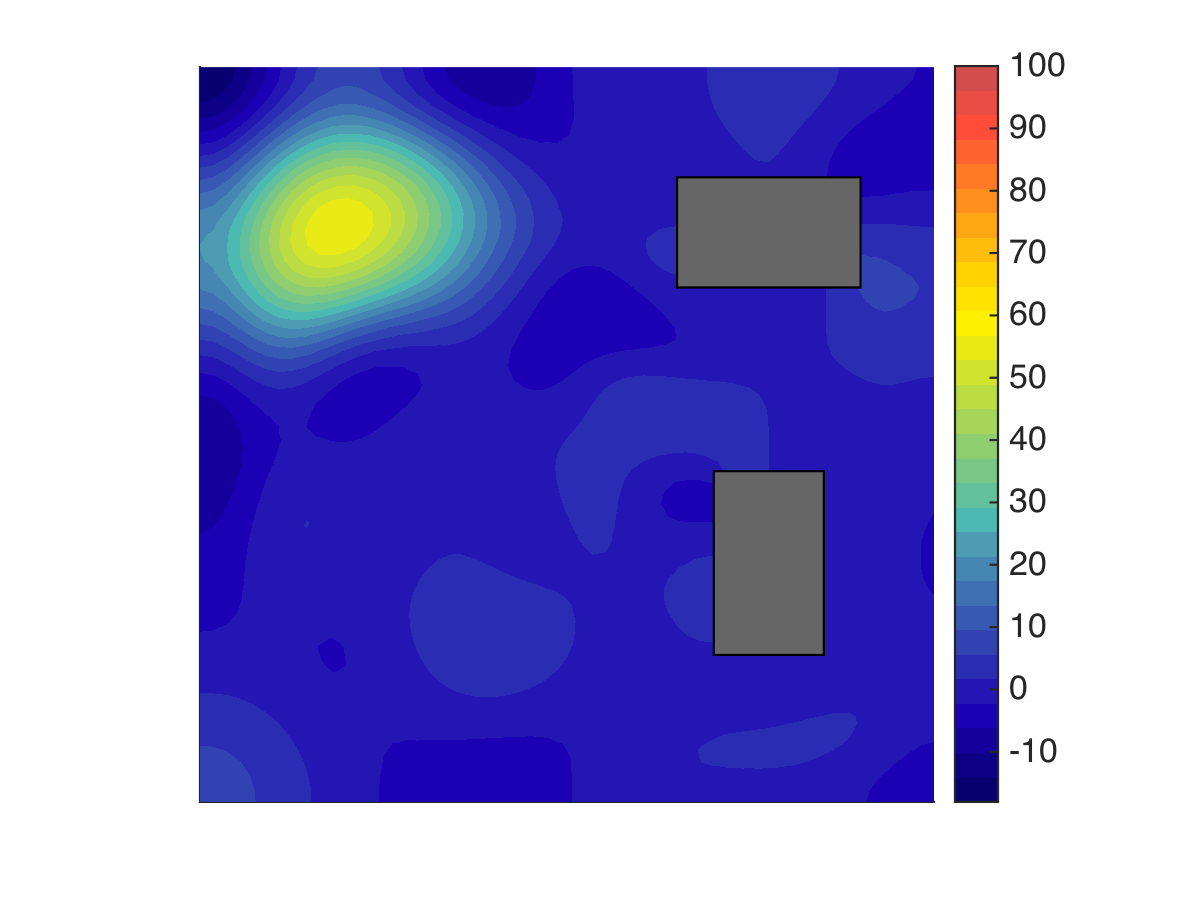}
\\
\caption{The true parameter (left) and the MAP estimator (right).} 
\label{fig:reconstruct}
\end{figure}

\begin{figure}[!ht]\centering
\includegraphics[width=.4\textwidth]{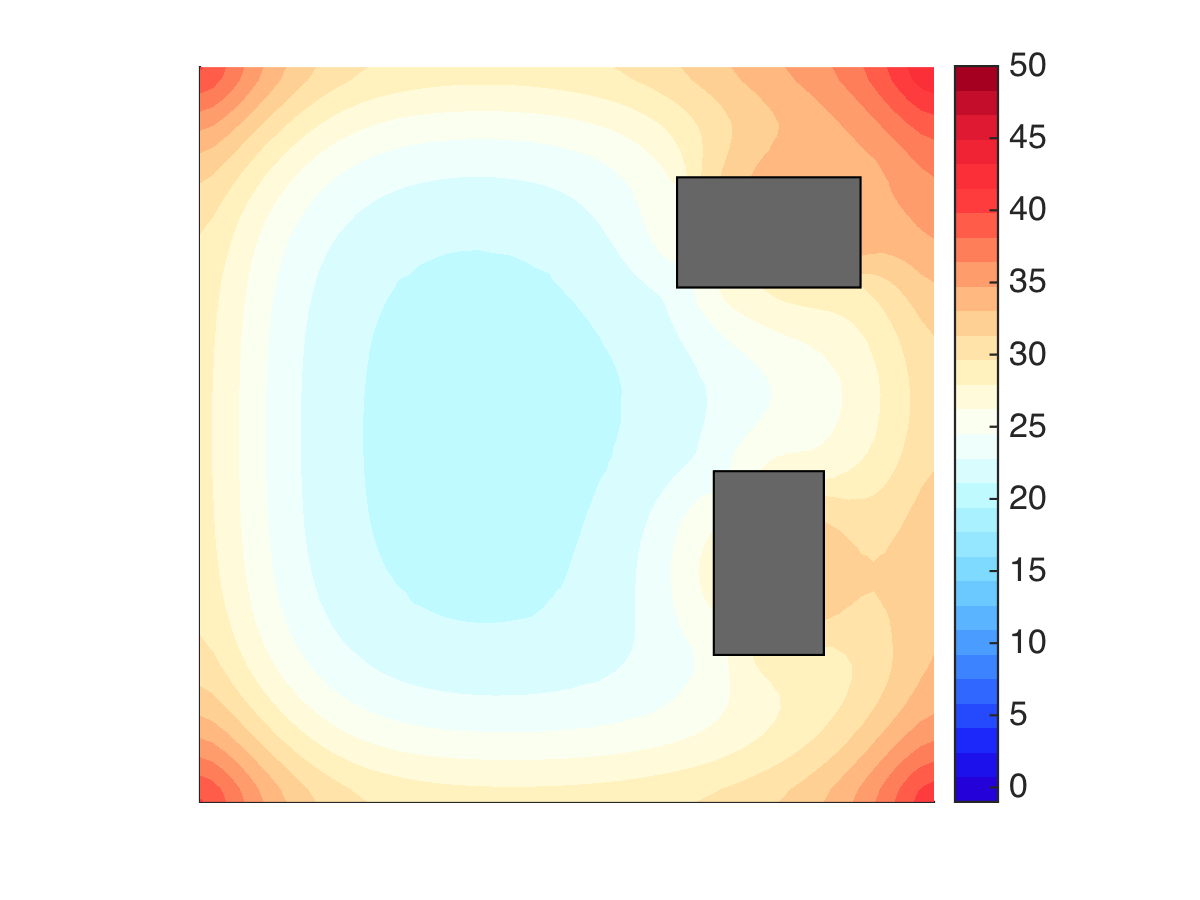}
\includegraphics[width=.4\textwidth]{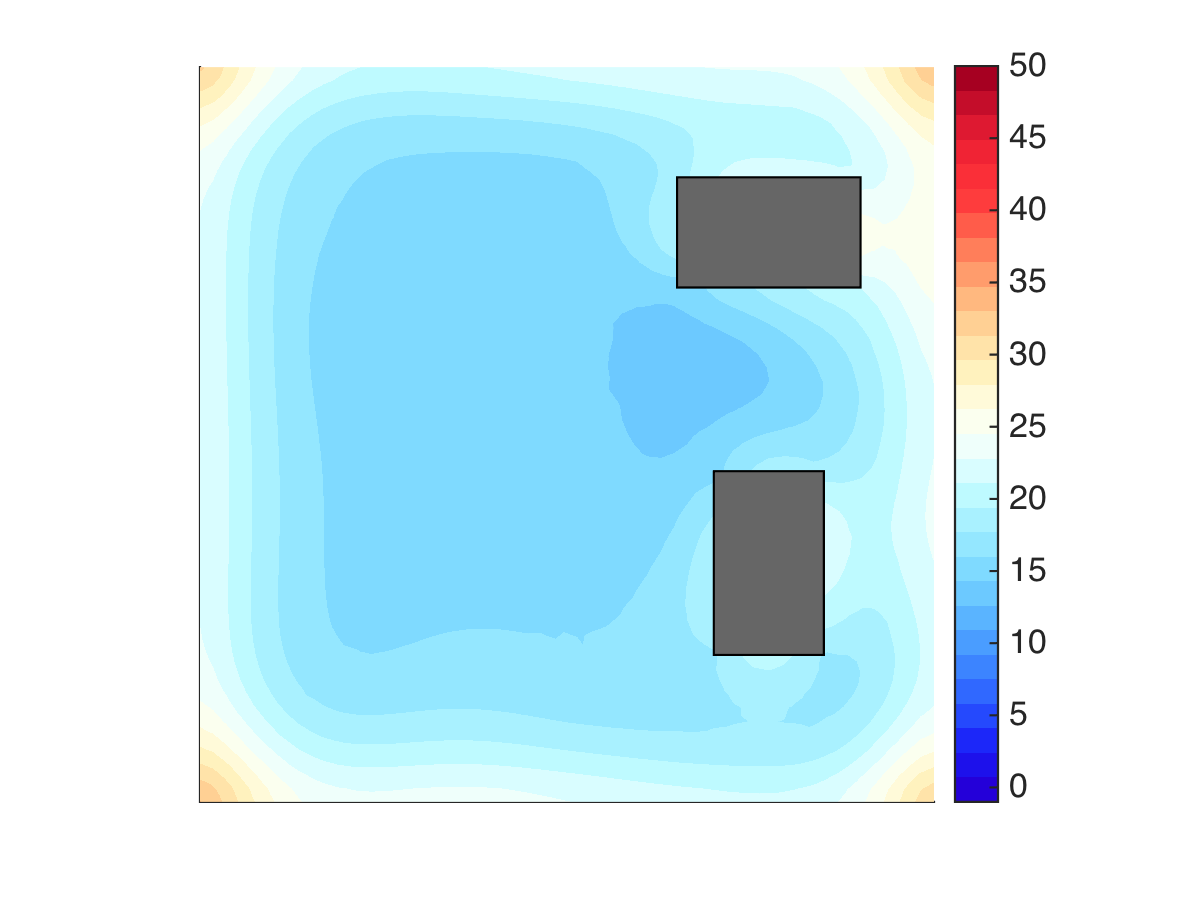}
\\
\includegraphics[width=.32\textwidth]{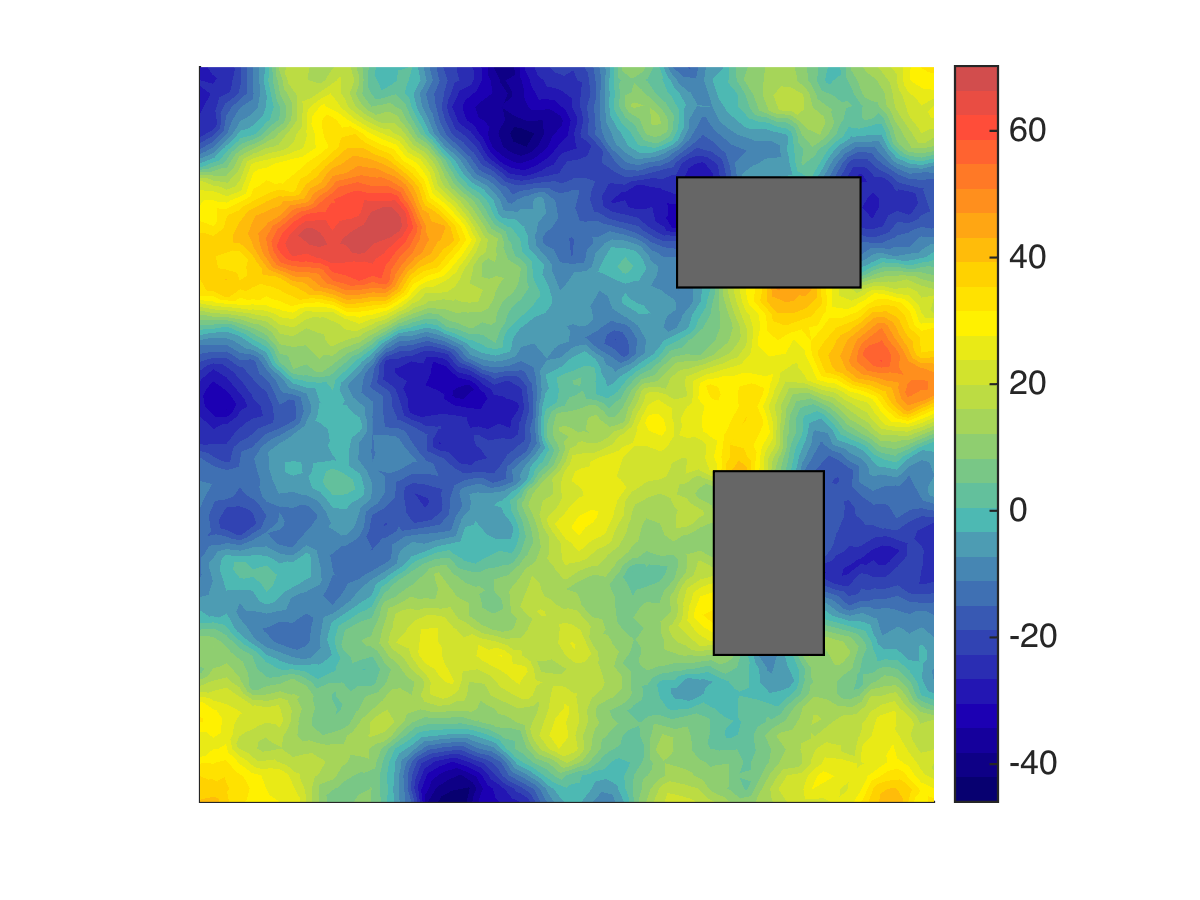}
\includegraphics[width=.32\textwidth]{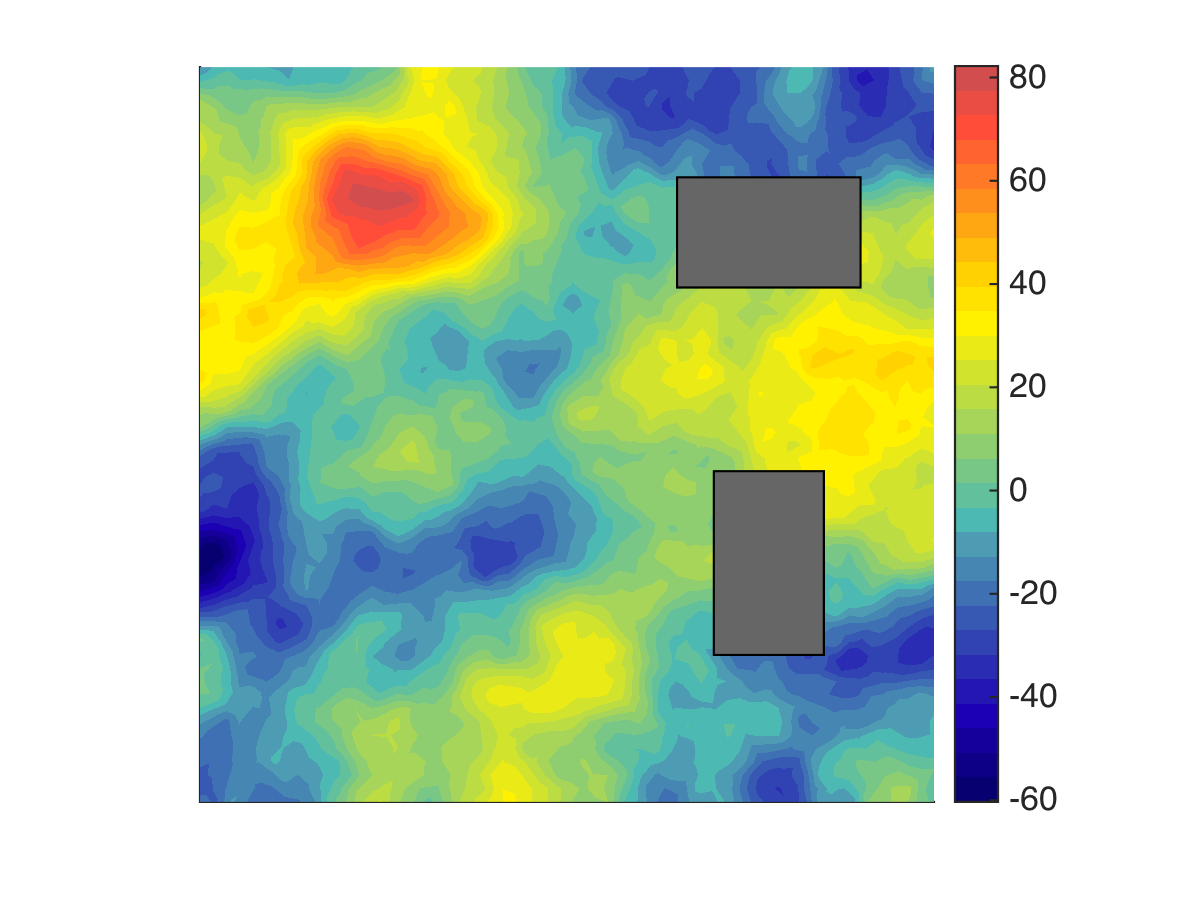}
\includegraphics[width=.32\textwidth]{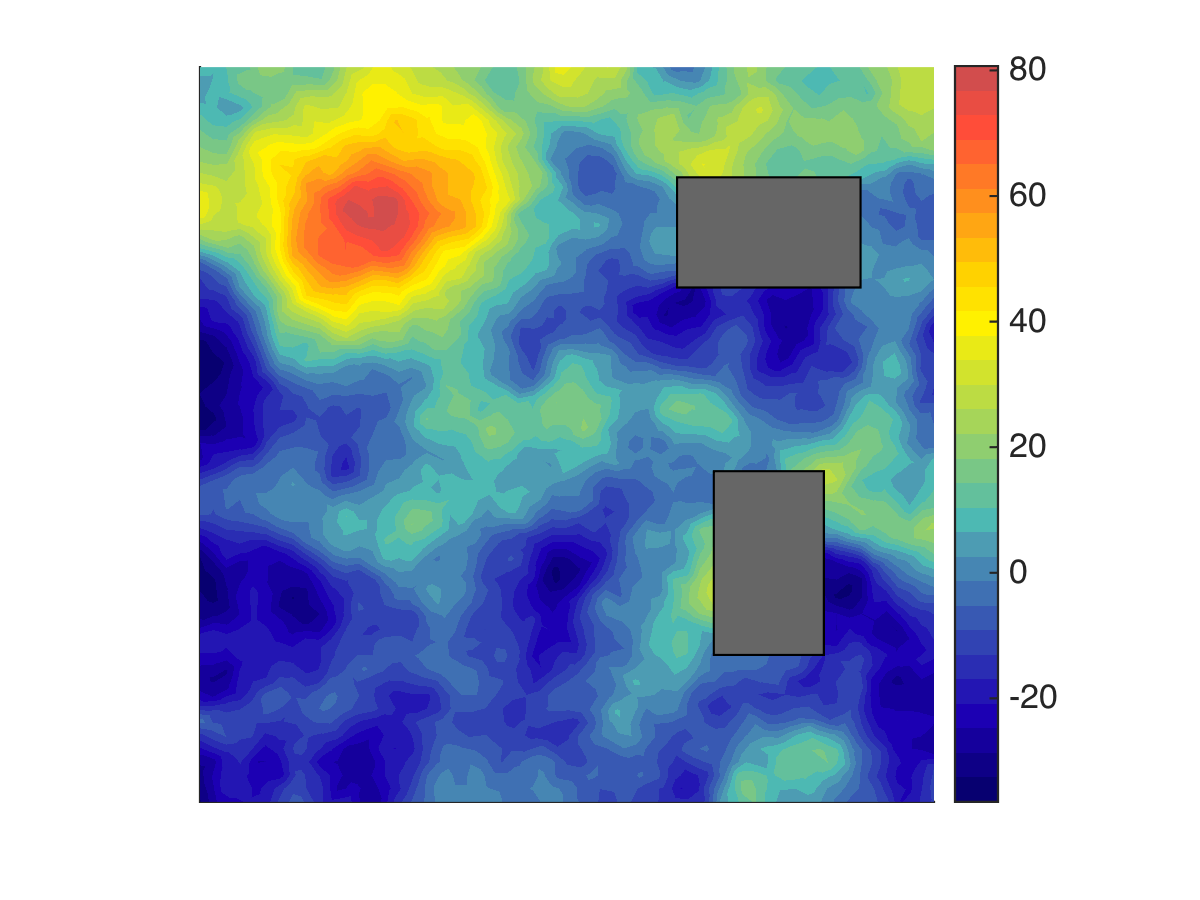}
\caption{The prior (left) and posterior (right) standard deviation fields.}
\label{fig:UQ}
\end{figure}

%
% Conclusions
% 
\section{Conclusion}\label{sec:conclusions}
%\arvind{Work in progress:}

We have developed a computational framework for D-optimal experimental design
for PDE-based Bayesian linear inverse problems with infinite-dimensional
parameters.  Our methods exploit low-rank structure of the
prior-preconditioned data misfit Hessian, and the
prior-preconditioned forward operator for fast estimation of 
OED objective and its gradient, as well as the KL divergence from posterior to 
prior.   Our developments resulted in three different approaches:
the low-rank spectral decomposition approach, the randomized approach,
and the approach based on fixed low-rank SVD of the prior-preconditioned forward
operator. Our numerical results, aiming at 
D-optimal sensor placement for initial state inversion in a time-dependent 
advection-diffusion equation, demonstrate the effectiveness of our
methods.

%We formulated the sensor placement problem as a D-optimal experimental design
%problem. The computational cost of a naive approach is prohibitively expensive.
%To tackle this, we proposed a two pronged approach that combined a
%gradient-based optimization solver with a continuation approach for sparsifying
%weights. Specifically, to reduce the cost of computing the objective function
%and the gradient we presented three different computational techniques.
%Together, these form the backbone of a computational framework for large-scale
%D-optimal designs. 

The randomized and spectral approaches provide accurate and efficient
estimators of the OED objective and gradient; the 
%were computed in exactly and the 
accuracy is controlled by setting the target rank $k$, which is 
specified a priori. 
%Empirical evidence is presented that shows that the optimal designs are
%comparable with the ``exact'' approach.  
Some open questions remain: can we adaptively determine the target rank $k$
during the optimization iterations so as to ensure the objective function and
the gradient are accurate to a given tolerance, with as few PDE solves as necessary?  
Can suitable error estimators
be incorporated into an optimization routine with theoretical
guarantees of convergence? Can the presented strategy be extended to OED for
nonlinear inverse problems? Addressing these questions is subject of our future
work. 

%Our hope is to address some of these issues in forthcoming
%publications. 

\section*{Acknowledgments}
This material was based upon work partially supported by the National Science
Foundation under Grant DMS--1127914 to the Statistical and Applied Mathematical
Sciences Institute. Any opinions, findings, and conclusions or recommendations
expressed in this material are those of the author(s) and do not necessarily
reflect the views of the National Science Foundation.”

\bibliography{ncsu_refs}
\bibliographystyle{siam}

\appendix
\newcommand{\noiseinv}{\ncov^{-1}}
\renewcommand{\S}{\mat{S}}
\section{Proofs}
\subsection{Proof of Theorem~\ref{thm:doptimal_criterion}}\label{ssec:doptproof}
We define the operator $\S =  (\HMpd + \mat{I})^{-1}$.
Using the definition of $\dpostmean$ 
we have,
\[
   \begin{aligned}
     \avey{ \cipfd{\dpostmean}{\dpostmean}} 
     &= \avey{\mip{\postcov\FF^*\noiseinv \obs}{\priorcov^{-1} (\postcov\FF^*\noiseinv\obs)}}\\
     &= \avey{ \eip{\obs}{\noiseinv \FF \postcov \priorcov^{-1} \postcov\FF^*\noiseinv\obs}}\\
     &= \avey{ \eip{\obs}{\noiseinv \FF \priorcov^{1/2} \S^2 \priorcov^{1/2}\FF^*\noiseinv\obs}},
   \end{aligned}
\]
where we have used 
that $\postcov = \priorcov^{1/2} \S \priorcov^{1/2}$ in the last step.
Due to our choice of the noise model, $\obs | \dpar \sim \GM{\FF \dpar}{\ncov}$, therefore,
by the formula for the expectation of a quadratic form we have,
\begin{equation}\label{equ:innerexp}
\begin{aligned}
&\avey{ \cipfd{\dpostmean}{\dpostmean}} =\\
%\avey{ \obs^T \noiseinv \FF \priorcov^{1/2} \S^2 \priorcov^{1/2}\FF^*\noiseinv\obs}\\
&\quad\trace(\noiseinv \FF \priorcov^{1/2} \S^2 \priorcov^{1/2}\FF^*)
+ \eip{\FF \dpar}{\noiseinv \FF \priorcov^{1/2} \S^2 \priorcov^{1/2}\FF^*\noiseinv \FF \dpar}.
\end{aligned}
\end{equation}
Now, the first term, involving the trace, simplifies as follows,
\begin{equation}\label{equ:first}
    \trace(\noiseinv \FF \priorcov^{1/2} \S^2 \priorcov^{1/2}\FF^*) = 
    \trace(\priorcov^{1/2} \S^2 \priorcov^{1/2}\FF^*\noiseinv\FF) = 
    \trace(\S^2 \HMpd).
\end{equation}
Next, we consider the second term in~\eqref{equ:innerexp}, which we average over  
the prior measure: 
\begin{equation*}
\begin{aligned}
\avep{
\eip{\FF \dpar}{\noiseinv \FF \priorcov^{1/2} \S^2 \priorcov^{1/2}\FF^*\noiseinv \FF \dpar}
}
&=\avep{\mip{ \dpar}{\FF\noiseinv \FF \priorcov^{1/2} \S^2 \priorcov^{1/2}\FF^*\noiseinv \FF \dpar}}\\
&=\avep{\mip{\dpar}{\HMd \priorcov^{1/2} \S^2 \priorcov^{1/2}\HMd \dpar}} \\
&= \trace(\HMd \priorcov^{1/2} \S^2 \priorcov^{1/2}\HMd\priorcov) \\
&= \trace(\HMpd \S^2 \HMpd) 
= \trace(\S^2 \HMpd^2).
\end{aligned}
\end{equation*}
This, along with~\eqref{equ:first} and~\eqref{equ:innerexp}, lead to 
\begin{equation}\label{equ:final}
\begin{aligned}
\avep{\avey{ \cipfd{\dpostmean}{\dpostmean}}} 
&=\trace(\S^2 \HMpd + \S^2 \HMpd^2)     \\     
&= \trace\big(\S^2 \HMpd (\mat{I} + \HMpd)\big)   
= \trace(\S^2 \HMpd \S^{-1})               \\ 
&= \trace(\S \HMpd) = \trace(\HMpd(\HMpd + \mat{I})^{-1}). 
\end{aligned}
\end{equation}
The stated result follows by averaging~\eqref{equ:DKL-fd-alt} over the prior
and noise, as in the statement of the theorem, and using~\eqref{equ:final}.\hfill\qedhere

\subsection{Proof of Theorem~\ref{p_kld}}\label{ssec:kld}

Subtracting~\eqref{equ:DKL-fd-approx} from~\eqref{equ:DKL-fd}, and applying the triangle inequality, we obtain 
\begin{equation}\label{e_kl_inter} E_\text{KL} \leq \> \frac12 \left[ |\logdet(\mat{I} + \HMs) -  \logdet ( \mat{I} + \mat{T})| + | \trace(\HMs(\mat{I}+\HMs)^{-1}) - \trace(\mat{T}(\mat{I} + \mat{T})^{-1}) | \right].\end{equation} 
By linearity of expectation, we tackle each term individually. Applying~\cite[Theorem 1]{SaibabaAlexanderianIpsen17}, we obtain 
\begin{equation}\label{e_loge} \expect[\mat\Omega]{|\logdet(\mat{I} + \HMs) -  \logdet ( \mat{I} + \mat{T})|} \leq \> \logdet(\mat{I}+\mat{\Lambda}_2) + \logdet(\mat{I}+\gamma^{2q-1}C_{\rm ge}\mat{\Lambda}_2). \end{equation}
For the second term, label all the eigenvalues of $\HMs$ as $\{\lambda_i\}_{i=1}^n$ and all the eigenvalues of $\mat{T}$ as $\{\tilde{\lambda}_i\}_{i=1}^\ell$. Then 
\begin{align*} \trace(\HMs(\mat{I}+\HMs)^{-1}) - \trace(\mat{T}(\mat{I} + \mat{T})^{-1}) = & \> \sum_{i=1}^n \frac{\lambda_i}{1+\lambda_i} -   \sum_{i=1}^\ell\frac{\tilde\lambda_i}{1+\tilde\lambda_i}\\
= & \> \sum_{i=1}^\ell \frac{\lambda_i - \tilde\lambda_i}{(1+\lambda_i)(1+\tilde\lambda_i) } + \sum_{i=\ell+1}^n\frac{\lambda_i}{1+\lambda_i}.
\end{align*}
From Cauchy interlacing theorem, 
$\lambda_i \geq \tilde\lambda_i \geq 0$ (see~\cite[Lemma 1]{SaibabaAlexanderianIpsen17}) so this expression is 
non-negative. We have the following inequalities 
\begin{align}\label{eqn:ineq1}
\frac{\lambda_i - \tilde\lambda_i}{(1+\lambda_i)(1+\tilde\lambda_i) }    \leq & \> \lambda_i - \tilde\lambda_i & i=1,\dots,\ell\\\label{eqn:ineq2}
\frac{\lambda_i}{1+\lambda_i}  \leq & \>  \lambda_i & i=\ell + 1,\dots,n.
\end{align}  
With these relations, then 
\[ \trace(\HMs(\mat{I}+\HMs)^{-1}) - \trace(\mat{T}(\mat{I} + \mat{T})^{-1}) \leq \> \sum_{i=1}^n\lambda_i - \sum_{i=1}^\ell \tilde\lambda_i = \trace(\HMs) - \trace(\mat{T}).\]
Since both sides of the inequality are nonnegative, we can take absolute values. Then apply~\cite[Theorem 1]{SaibabaAlexanderianIpsen17} to obtain 
\begin{equation}\label{e_tre} \expect[\mat\Omega]{|\trace(\HMs(\mat{I}+\HMs)^{-1}) - \trace(\mat{T}(\mat{I} + \mat{T})^{-1})| } \leq \> (1+C_{\rm ge}\gamma^{2q-1}) \trace(\mat\Lambda_2).\end{equation}
Substitute~\eqref{e_loge} and~\eqref{e_tre} into~\eqref{e_kl_inter} to complete the proof.\hfill\qedhere

\subsection{Proof of Theorems~\ref{prop:err_exactk} and~\ref{p_grad}}\label{ss_grad}
First, we 
record the following basic lemma:
\begin{lemma}\label{lem:basic}
Let $\mat{A}$ be an $n \times n$ matrix and suppose $\mat{B}$ is a symmetric positive
semidefinite matrix. Then, we have
$|\trace(\mat{A}\mat{B})| \leq \norm{\mat{A}}_2\trace(\mat{B})$.
\end{lemma}
\begin{proof}
Let $\{ \vec{e}_j \}_{j = 1}^n$ be the orthonormal basis of eigenvectors of $\mat{B}$ with
corresponding  (real, non-negative) eigenvalues) $\{ \lambda_j \}_{j = 1}^n$. 
By Cauchy-Schwartz inequality,
\[
|\trace(\mat{A}\mat{B})| = |\sum_j \ip{\vec{e}_j}{\mat{A}\mat{B}\vec{e}_j}|
           \leq \sum_j \norm{\vec{e}_j}_2 \norm{\mat{A}}_2 \norm{\mat{B}\vec{e}_j}_2 
           = \norm{\mat{A}}_2 \sum_j \lambda_j
           = \norm{\mat{A}}_2 \trace(\mat{B}).
\]
%We have used Cauchy-Schwartz inequality in step 2. 
\end{proof}

\begin{proof}[Theorem~\ref{prop:err_exactk}]
The first statement is immediate; the second one follows from 
\[
|\trace( (\mat{UDU}\tran) \mat{Z}_j) - 
\trace( (\mat{U}_k \mat{D}_k \mat{U}_k\tran) \mat{Z}_j)|
\leq \norm{\mat{Z}_j} \sum_{j = k+1}^K \frac{\lambda_i}{1+\lambda_i},
\]
where we have abbreviated $\mat{Z}_j \equiv \mat{M}^{1/2}\FFp^*
\mat{E}^\sigma_j \FFp\mat{M}^{-1/2}$. The inequality is justified in
Lemma~\eqref{lem:basic} in the Appendix.
\end{proof}

\begin{proof}[Theorem~\ref{p_grad}]
From Lemma~\ref{lem:basic}
\[ |\partial_j J(\vec{w})- \widehat{\partial_j J}_\text{rand}(\vec{w})| \> \leq \> \| \mat{Z}_j\|_2 |\trace(\HMs(\mat{I}+\HMs)^{-1}) - \trace(\mat{T}(\mat{I} + \mat{T})^{-1})|.\]
Taking expectations and use~\eqref{e_tre} to obtain the desired bound.
 %Apply~\cite[Theorem
%1]{SaibabaAlexanderianIpsen17} to Proposition~\ref{p_gradd}. 
\end{proof}

\subsection{Proof of Theorem~\ref{thm:froz}}\label{apdx:froz}
Before we turn to the proof of Theorem~\ref{thm:froz}, we present the following Lemma. 
\begin{lemma}\label{lemma:logdet} Let $\mat{M},\mat{N} \in \mathbb{C}^{n\times n}$ be Hermitian positive 
semidefinite with $\mat{M} \geq \mat{N}$. Then 
	\[ 0 \leq \logdet(\mat{I}+\mat{M}) - \logdet(\mat{I}+\mat{N}) \leq \logdet(\mat{I} + \mat{M} - \mat{N}).\]
\end{lemma}
\begin{proof} The inequality $\mat{M} \geq \mat{N}$ is to be interpreted as $\mat{M} - \mat{N}$ is Hermitian positive semidefinite. Write $\mat{X} = (\mat{I}+\mat{N})^{-1/2}$. Then, by multiplicativity of determinants, 
	\[\logdet(\mat{I}+\mat{M}) - \logdet(\mat{I}+\mat{N}) = \logdet \mat{X}(\mat{I}+\mat{M})\mat{X}.\] 
	Write $\mat{M} = \mat{E} + \mat{N}$, with $\mat{E} = \mat{M} - \mat{N}$ and  
	\[ \mat{X}(\mat{I} + \mat{M})\mat{X} = \mat{X}(\mat{I}+ \mat{N})\mat{X} + \mat{XEX} = \mat{I} + \mat{XEX}.\]
	Since $\mat{E}$ is hpsd, it has a well-defined square root. Note that $\mat{X}$ has singular values at most $1$, and is a contraction matrix. Furthermore, the multiplicative singular value inequalities~\cite[page~188]{Bhatia09} imply
	\[ \lambda_j(\mat{XEX}) = \sigma_j^2 (\mat{E}^{1/2}\mat{X})\leq \sigma_j^2(\mat{E}^{1/2}) = \lambda_j(\mat{E}) \qquad j=1,\dots,n.\]
	Writing the determinant as the product of the eigenvalues
	\[ \det(\mat{I}+\mat{XEX}) = \prod_{i=1}^n (1 + \lambda_i(\mat{XEX})) \leq \prod_{i=1}^n (1 + \lambda_i(\mat{E})) =  \det(\mat{I} + \mat{E}).\]
Taking logarithms delivers the desired upper bound. The lower bound
follows from the fact that since $\mat{E}$ is positive semidefinite, and
$\mat{X}$ is Hermitian, $\mat{I} + \mat{XEX} \geq \mat{I}$. Therefore, 
$\logdet(\mat{I} + \mat{XEX}) \geq \logdet \mat{I} = 0$.
\end{proof}

%\begin{theorem} Let $\mat{$
	\begin{proof}[Theorem~\ref{thm:froz}]
	Partition the singular value decomposition as  $\FFp = \FFp_k + \FFp_{k,\perp}$, and observe 
	\[ \FFp\FFp^* = \FFp_k\FFp_k^* + \FFp_{k,\perp} \FFp_{k,\perp}^*.\]
	The determinant identity~\cite[Corollary 2.1]{Ou81} gives 
	\begin{equation}\label{eqn:detiden}
	\log\det(\mat{I} + \FFp^*\mat{W}\FFp) = \logdet(\mat{I} + \mat{W}^{1/2}\FFp\FFp^*\mat{W}^{1/2}).
	\end{equation}
	It can be readily verified that $\FFp\FFp^* \geq \FFp_k\FFp_k^*$; therefore, by the lower bound in Lemma~\ref{lemma:logdet}
	\[ |J(\vec{w}) - \widehat{J}_\text{froz}(\vec{w})| = \log\det(\mat{I} + \FFp^*\mat{W}\FFp)-\log\det(\mat{I} +\FFp_k^*\mat{W}\FFp_k ).  \]

		Use~\eqref{eqn:detiden} along with the upper bound of Lemma~\ref{lemma:logdet}	
	\begin{align*} \log\det(\mat{I} + \FFp^*\mat{W}\FFp) - \log\det(\mat{I} +\FFp_k^*\mat{W}\FFp_k ) \leq& \>  \logdet(\mat{I} + \mat{W}^{1/2}\FFp_{k,\perp}\FFp_{k,\perp}^*\mat{W}^{1/2}) \\ 
		 = & \>  \logdet(\mat{I} + \FFp_{k,\perp}^*\mat{W}\FFp_{k,\perp}).
	\end{align*}
	Since the weights are at most $1$, then $\mat{W} \leq \mat{I}$ and $$\logdet(\mat{I} + \FFp_{k,\perp}^*\mat{W}\FFp_{k,\perp}) \leq \logdet(\mat{I} + \FFp_{k,\perp}^*\FFp_{k,\perp}) = \logdet(\mat{I} + \mat{\Sigma}_2^2).$$
		The inequality is due to the properties of L{o}ewner partial ordering of Hermitian matrices~\cite[Corollary 7.7.4]{HoJ13}. The final equality employs the determinant identity. 
	\end{proof}

\end{document}